\documentclass{amsart}

\usepackage{amsfonts}
\usepackage{amssymb}
\usepackage{tikz}
\usepackage{afterpage}

\usepackage{euscript}

\usetikzlibrary{arrows,decorations.pathreplacing}
\usetikzlibrary{shapes,decorations.pathmorphing}
\usetikzlibrary{arrows,calc} 
 \usetikzlibrary{cd}

\allowdisplaybreaks

\usepackage{float}

\usepackage{graphicx}
\usepackage{bm}
\usepackage[
	pdftitle={},
	pdfauthor={},
	ocgcolorlinks,
	linkcolor=linkblue,
	citecolor=linkred,
	urlcolor=linkred]
{hyperref}
\usepackage{color}
\definecolor{linkred}{rgb}{0.75,0,0}
\definecolor{linkblue}{rgb}{0,0,0.75}
\usepackage{microtype}
\usepackage{multicol}
\usepackage{booktabs}
\usepackage{latexsym}
\usepackage{multirow}
\usepackage{setspace}
\usepackage{enumerate}

\theoremstyle{plain}
\newtheorem{theorem}{Theorem}
\newtheorem{exercise}{Exercise}
\newtheorem{proposition}{Proposition}[section]
\newtheorem{thm}[proposition]{Theorem}

\theoremstyle{definition}

\newtheorem{definition}[proposition]{Definition}

\newcommand{\beq}{\begin{equation}}
\newcommand{\eeq}{\end{equation}}

\newcommand{\cal}{\mathcal}

\newcommand{\cc}{\EuScript{C}}
\newcommand{\cl}{\EuScript{L}}
\newcommand{\cd}{\mathcal{D}}
\newcommand{\ce}{\mathcal{E}}
\newcommand{\cf}{{\cal F}}
\newcommand{\ch}{\mathcal{H}}
\newcommand{\co}{\mathcal{O}}

\newcommand{\ct}{\mathcal{T}}
\newcommand{\cu}{\mathcal{U}}
\newcommand{\cv}{\mathcal{V}}

\newcommand{\bc}{\mathbb{C}}
\newcommand{\be}{\mathbb{E}}
\newcommand{\bh}{\mathbb{H}}
\newcommand{\bn}{\mathbb{N}}
\newcommand{\bp}{\mathbb{P}}
\newcommand{\bq}{\mathbb{Q}}
\newcommand{\br}{\mathbb{R}}
\newcommand{\bz}{\mathbb{Z}}
\newcommand{\tr}{\text{tr}\hspace{.5mm}}

\newcommand{\modm}{\cal M}

\begin{document}
	
\title{Spin structures and measures on the moduli space of curves}
\author{Paul Norbury}
\address{School of Mathematics and Statistics, University of Melbourne, VIC 3010, Australia}
\email{\href{mailto:norbury@unimelb.edu.au}{norbury@unimelb.edu.au}}
\thanks{}
\subjclass[2020]{32G15; 14H10; 30F35}
\date{\today}

\begin{abstract}
This article arose out of notes for a CIME summer school mini-course in Cetraro.  In it we  develop differential-geometric constructions on the moduli space
${\cal M}_{g,n}^{\rm spin}$ of smooth spin curves, with particular emphasis on
their interpretation in hyperbolic geometry.  We relate the locally constant
sheaf arising from a hyperbolic spin structure to the corresponding
holomorphic sheaf, extending this known relationship to spin curves with Ramond
marked points.  We use the hyperbolic description to construct characteristic
forms and spin measures on moduli spaces of hyperbolic surfaces with geodesic
boundary, providing a framework for applying Teichm\"uller-theoretic methods
to their volumes and volume recursions.  Finally, we study separately the
even and odd components of spin moduli space and find symmetries between their
volume contributions which are not apparent from the contributions of
individual boundary strata.
\end{abstract}

\maketitle

\tableofcontents

\section{Introduction}

The Weil--Petersson volumes of the moduli spaces 
${\cal M}_{g,n}$ parametrising genus $g$ curves with $n$ labeled points $(C, p_1,\ldots,p_n)$,
for any $g,n \in\bn=\{0,1,2,\ldots\}$ satisfying $2g - 2 + n > 0$ play
a central role in both pure mathematics and theoretical physics.
Their study has led to deep and productive connections between the two areas.
In particular, their appearance in two-dimensional quantum gravity implies
that the partition functions built from these volumes satisfy remarkable
integrability properties. From the algebro-geometric viewpoint, these volumes
arise as intersection numbers on the moduli spaces $\overline{{\cal M}}_{g,n}$ of
genus $g$ stable curves.

Recently, the super Weil--Petersson volumes of the moduli spaces of
genus $g$ super Riemann surfaces with $n$ labeled points, denoted $\widehat{{\cal M}}_{g,n}$,
have been shown to possess analogous structures. They satisfy a Virasoro 
constraint, realised through a recursion that parallels Mirzakhani's 
recursion for classical volumes \cite{MirSim}.
This super-recursion was originally described heuristically using techniques 
from super geometry \cite{SWiJTG} and later proven rigorously using
algebro-geometric methods in \cite{NorEnu}.

The moduli space of spin curves ${\cal M}_{g,n}^{\rm spin}$ parametrising genus curves $(C, p_1,\ldots,p_n)$ equipped with a spin structure appears
naturally in this context as the underlying topological space of the moduli
space of super Riemann surfaces $\widehat{{\cal M}}_{g,n}$, an infinitesimal fattening of ${\cal M}_{g,n}^{\rm spin}$.
In these notes we work solely with this topological space, rather than the
full supermanifold. A central object of study is a naturally defined vector
bundle
\[
E_{g,n} \longrightarrow{\cal M}_{g,n}^{\rm spin},
\]
arising via a cohomological construction in Section~\ref{vecE}.  It is naturally the normal bundle of
${\cal M}_{g,n}^{\rm spin}$ inside  $\widehat{{\cal M}}_{g,n}$. For this reason, $E_{g,n}$ exhibits properties analogous to those of
the tangent bundle $T{\cal M}_{g,n}$ of the classical moduli space of curves.  In particular, a natural Hermitian metric is constructed on $E_{g,n}$---see~\eqref{hermetric}---using the canonical complete hyperbolic surface uniformising $C-\{p_1,...,p_n\}$, directly mirroring the Weil--Petersson Hermitian metric on $T{\cal M}_{g,n}$.  The Hermitian metrics on $T{\cal M}_{g,n}$ and $E_{g,n}$ give rise to the Weil--Petersson symplectic form $\omega_{WP}$, respectively a natural Euler form $e(E_{g,n})$, which together define a finite measure on ${\cal M}_{g,n}^{\rm spin}$
 \[\mu=e(E_{g,n})\exp\omega_{WP}\in\Omega^{\text{top}}{\cal M}^{\rm spin}_{g,n}.\]
This pushes forward to a measure on ${\cal M}_{g,n}$, which decompose into a sum of measures obtained via restriction of $\mu$ to each connected component of ${\cal M}^{\rm spin}_{g,n}$.  

The Weil--Petersson form $\omega_{WP}$ arises naturally via the realisation of ${\cal M}_{g,n}$ as the moduli space of complete hyperbolic surfaces with cusps at the labeled points.  For any $(L_1,...,L_n)\in\br_{\geq0}$ there is a deformation $\omega_{WP}(L_1,...,L_n)$ of $\omega_{WP}$ obtained by deforming the cusps 
to geodesic boundary components.  This defines $(L_1,...,L_n)$ dependent measures
\[\mu(L_1,...,L_n)=e(E_{g,n})\exp\omega_{WP}(L_1,...,L_n)\in\Omega^{\text{top}}{\cal M}^{\rm spin}_{g,n}.\]
Spin structures on the complete hyperbolic surface 
$C-\{p_1,\ldots,p_n\}$ that come from the restriction of a spin structure on the hyperbolic surface $C$ form an open and closed subset ${\cal M}_{g,\{1^n\}}^{\rm spin}\subset{\cal M}_{g,n}^{\rm spin}$ known as the {\em Neveu--Schwarz} components.  The {\em Neveu--Schwarz} components consist of two connected components corresponding to even and odd spin structures.    
The total measure
\[ \widehat{V}_{g,n}(L_1,...,L_n):=(-1)^n2^{g-1+n}\int_{{\cal M}_{g,\{1^n\}}^{\rm spin}}\mu(L_1,...,L_n)
\]
arises as the super Weil--Petersson volume of the Neveu--Schwarz components of the moduli space of super Riemann surfaces $\widehat{{\cal M}}_{g,n}$.  The genus zero Neveu--Schwarz volumes vanish $\widehat{V}_{0,n}(L_1,...,L_n)=0$.   Stanford and Witten \cite{SWiJTG} showed that these volumes satisfy recursion relations analogous to Mirzakhani's recursion relations \cite{MirSim} for the Weil--Petersson volumes of ${\cal M}_{g,n}$.  Their proof, using supergeometry, is heuristic and the recursion was proven rigorously via algebro-geometric techniques in \cite{NorEnu}.
Define
\begin{equation}  \label{kerD} 
D(x,y,z)=\frac{\sinh\frac{x}{4}\sinh\frac{y+z}{4}}{\cosh\frac{x-y-z}{4}\cosh\frac{x+y+z}{4}}
\end{equation}
and $R(x,y,z)=\tfrac12(D(x+y,z,0)+D(x-y,z,0))$. 

\begin{theorem}[Stanford-Witten \cite{SWiJTG}, N. \cite{NorEnu}] \label{NSrec}
The volumes $\widehat{V}_{g,n}(L_1,...,L_n)$ for $g>0$ and $n>0$ are uniquely determined by the following recursion relations:
\begin{align*} 
L_1&\widehat{V}_{g,n}(L_1,...,L_n)=\tfrac18L_1\delta_{1,n}\delta_{1,g}+\sum_{j=2}^n\int_0^\infty \hspace{-2mm}xR(L_1,L_j,x)\widehat{V}_{g,n-1}(x,L_{K\backslash\{j\}})dx\\
+&\tfrac12\hspace{-1mm}\int_0^\infty\hspace{-3mm}\int_0^\infty\hspace{-4mm} xyD(L_1,x,y)\big[\widehat{V}_{g-1,n+1}(x,y,L_K)\hspace{-1mm}+\hspace{-4mm}\mathop{\sum_{g_1+g_2=g}}_{I \sqcup J = K}\hspace{-3mm}\widehat{V}_{g_1,|I|+1}(x,L_I)\widehat{V}_{g_2,|J|+1}(y,L_J)\big]dxdy
\nonumber
\end{align*}
for $K=(2,...,n)$.
\end{theorem}
From Theorem~\ref{NSrec} one can deduce that $\widehat{V}_{g,n}(L_1,...,L_n)$ is polynomial in the $L_i$ and further deduce underlying Virasoro structure and a relationship with the KdV hierarchy \cite{NorEnu}.   

The measures over general, non Neveu--Schwarz, components of ${\cal M}_{g,n}^{\rm spin}$ produce volumes $\widehat{V}_{g,n}^{(m)}(L_1,...,L_n)$ defined in \eqref{voldef}, where $m$ denotes the number of Ramond points---see Section~\ref{boundbeh}---so $\widehat{V}_{g,n}(L_1,...,L_n)=\widehat{V}_{g,n}^{(0)}(L_1,...,L_n)$.  They assemble to produce the following series in $s$:
\[
\widehat{V}_{g,n}(s,L_1,...,L_n):=\sum_{m=0}^\infty\frac{s^m}{m!}\widehat{V}_{g,n}^{(m)}(L_1,...,L_n).
\]
The factor $1/m!$ makes the $m$ Ramond points unlabeled. The genus zero volumes are now non-trivial, and in particular the disk and annulus functions $\widehat{V}_{0,1}$ and $\widehat{V}_{0,2}$ arise from hyperbolic surfaces due to the presence of Ramond points.

Remarkably, $\widehat{V}_{g,n}(s,L_1,...,L_n)$ satisfies the same recursion relations from Theorem~\ref{NSrec} satisfied by $\widehat{V}_{g,n}(L_1,...,L_n)$, with different initial conditions.  
\begin{theorem}[Alexandrov, N., \cite{ANoSup}]  \label{thmsup}
The volumes $\widehat{V}_{g,n}(s,L_1,...,L_n)$ for $g\geq0$ and $n>0$ are uniquely determined by the Stanford-Witten recursion:
\begin{align} \label{recsup}
L_1&\widehat{V}_{g,n}(s,L_1,...,L_n)=\tfrac12\int_0^\infty\hspace{-2mm}\int_0^\infty\hspace{-2mm} xyD(L_1,x,y)P_{g,n+1}(x,y,L_K)dxdy\\
&+\sum_{j=2}^n\int_0^\infty \hspace{-2mm}xR(L_1,L_j,x)\widehat{V}_{g,n-1}(s,x,L_{K\backslash\{j\}})dx+\delta_{1,n}(\tfrac{s^2\delta_{0,g}}{2}+\tfrac{\delta_{1,g}}{8})L_1
\nonumber 
\end{align}
for $K=(2,...,n)$ and
\[P_{g,n+1}(x,y,L_K)=\widehat{V}_{g-1,n+1}(s,x,y,L_K)+\hspace{-3mm}\mathop{\sum_{g_1+g_2=g}}_{I \sqcup J = K}\hspace{-2mm}\widehat{V}_{g_1,|I|+1}(s,x,L_I)\widehat{V}_{g_2,|J|+1}(s,y,L_J).
\]
\end{theorem}
The recursion relations \eqref{recsup} give a recursive procedure in increasing powers of $s$ which specialises at $s=0$ to Theorem~\ref{NSrec}.  Note that $D(x,y,z)$ and $R(x,y,z)$ used to define the recursion relations do not depend on $s$.
 
The restriction of \eqref{recsup} to the disk case $(g,n)=(0,1)$ produces the following elegant recursion relation.
\begin{theorem}[\cite{NorSup}]  \label{disk}
The disk function $\widehat{V}_{0,1}(s,L)$ is uniquely determined by the recursion relation
\[ \widehat{V}_{0,1}(s,L)=\frac{s^2}{2!}+\frac{1}{2L}\int_0^\infty\int_0^\infty\hspace{-2mm} xyD(L,x,y)\widehat{V}_{0,1}(s,x)\widehat{V}_{0,1}(s,y)dxdy.
\]
\end{theorem}
The disk recursion relation produces 
\[ \widehat{V}_{0,1}(s,L)=\frac{s^2}{2!}+\left(6\pi^2+\frac12L^2\right)\frac{s^4}{4!}+\left(330\pi^4 + 30\pi^2L^2 + \frac38L^4\right)\frac{s^6}{6!}+...
\]

The main focus of the paper is contained in Sections~\ref{spinstr}, \ref{sheafc} and \ref{spinmeas} which give differential-geometric constructions on the moduli space
${\cal M}_{g,n}^{\rm spin}$ of smooth spin curves.  They survey existing results,
with particular emphasis on constructions that admit a direct interpretation
in terms of hyperbolic geometry, while also developing several new aspects of
the theory.  One new contribution is Theorem~\ref{exactseq}, which extends to
Ramond points the relationship between the cohomology of a locally constant
sheaf arising from a hyperbolic spin structure and that of the corresponding
holomorphic sheaf.  This motivates the detailed treatment of sheaf cohomology
in Sections~\ref{spinstr} and \ref{sheafc}.

A second aim is to develop the spin measures directly on moduli spaces of
hyperbolic surfaces with geodesic boundary.  This provides a framework in
which Teichm\"uller-theoretic methods may be applied to the measures and their
volume recursions.  We also study separately the contributions of the even and
odd components of the moduli space of spin curves, revealing symmetries which
are not apparent from the individual boundary-stratum contributions.

We give less detail on algebro-geometric 
constructions, referring instead to the existing literature, and concentrate
on differential-geometric interpretations, examples, and open problems.
Nevertheless, several of the results ultimately rely on intersection theory
on the compactifications
$\overline{{\cal M}}_{g,n}$ and
$\overline{{\cal M}}_{g,n}^{\rm spin}$, so the necessary algebro-geometric
constructions are recalled in Section~\ref{compact}.

Most of the constructions in these notes can be formulated entirely within
ordinary geometry, without using supergeometry.  Supergeometry nevertheless
provides an important motivation: the heuristic derivation of the volume
recursion due to Stanford and Witten \cite{SWiJTG} uses the geometry of super
Riemann surfaces.  We describe this viewpoint in Section~\ref{super} and
explain how it is related to the differential-geometric constructions
developed here.

{\bf Acknowledgements.}  The author would like to thank Ga\"etan Borot for useful comments, the Centro Internazionale Matematico Estivo for hosting the summer school {\em Enumerative geometry, moduli spaces, and quantization} in Cetraro from which these notes arose, and the Max Planck Institute for Mathematics in the Sciences, Leipzig, for its hospitality while part of this work was carried out.
 ChatGPT was used during the preparation of this paper to provide editing support.

\section{Spin structures on Riemann surfaces} \label{spinstr}



\subsection{Spin structures}   

Let $C$ be a Riemann surface. Given coordinate charts
\[
z=\phi_U:U\to\mathbb C,
\qquad
w=\phi_V:V\to\mathbb C,
\]
the transition function on the overlap $U\cap V$ is the holomorphic
change of coordinates
\[
z=z(w)=\phi_U\circ\phi_V^{-1}.
\]
Since coordinate changes are conformal, their derivatives satisfy
$z'(w)\neq 0$.

The canonical bundle $\omega_C=T^*C$ is the holomorphic line bundle
whose transition functions are $z'(w)$. Equivalently, local sections
transform according to
\[
dz=z'(w)\,dw.
\]
\begin{definition}
A spin structure on $C$ is a holomorphic line bundle
$L\to C$ together with an isomorphism
\[
\phi:L^{\otimes 2}\stackrel{\cong} {\longrightarrow}\omega_C.
\]
\end{definition}
In terms of transition functions, a spin structure is obtained by
choosing square roots of the transition functions of $\omega_C$ 
\[
z'(w)^{1/2}
\]
that satisfy the cocycle
conditions on double and triple overlaps. Since $z'(w)\neq0$, a
holomorphic square root exists locally via a choice of branch.

 The choice of square root must be compatible on overlaps:  on $U\cap V$, we have $z'(w)w'(z)=1$, and we choose square roots so that $z'(w)^{\frac12}w'(z)^{\frac12}=1$, and on $U\cap V\cap W$, we have $z'(w)w'(u)u'(z)=1$, then we also choose square roots so that $z'(w)^{\frac12}w'(u)^{\frac12}u'(z)^{\frac12}=1$.
Given any Riemann surface, a spin structure exists, i.e. a compatible choice of square roots may be made.
 If
$L$ and $L'$ are two spin structures, then
\[
L'=L\otimes \eta,
\]
where $\eta$ is a line bundle satisfying
\[
\eta^{\otimes 2}\cong\mathcal O_C.
\]
Hence the set of spin structures is a torsor for $ H^1(C,\bz_2)$. 
In particular, a compact genus $g$ surface $C$ possesses $2^{2g}$ distinct spin structures.  

For a non-compact Riemann surface the line bundle $L$
is holomorphically trivial, but the square-root isomorphism still carries
non-trivial information, classified by $H^1(S,\mathbb Z_2)$.  

Next we describe a construction of a spin structure using a hyperbolic metric.  The conformal and hyperbolic constructions turn out to be equivalent.

\subsubsection{Hyperbolic spin structures}  Spin structures have a particularly nice description in terms of hyperbolic structures on surfaces.  A hyperbolic metric on an orientable surface $\Sigma$ determines a principal $SO(2)$ bundle $P_{SO}(\Sigma)$ given by the orthonormal frame bundle of $\Sigma$.   A {\em hyperbolic spin structure} on $\Sigma$ is a principal $SO(2)$ bundle double cover
\[\pi:P_{\text{Spin}}(\Sigma)\to P_{SO}(\Sigma),\qquad \pi(e^{i\theta}\cdot p)=e^{2i\theta}\cdot\pi(p).\]
Hence a hyperbolic spin structure can be naturally identified with an element 
\begin{equation}  \label{spinlift}
\eta\in H^1(P_{SO}(\Sigma),\bz_2)=\text{Hom}(\pi_1(P_{SO}(\Sigma)),\bz_2)
\end{equation} 
satisfying $r(\eta)\neq0$ in the exact sequence
\[0\to H^1(\Sigma,\bz_2)\to H^1(P_{SO}(\Sigma),\bz_2)\stackrel{r}{\to} H^1(SO(2),\bz_2)\to   0
\]
where the final arrow maps to $0=w_2\in H^2(\Sigma,\bz_2)$. The exact sequence shows that any two hyperbolic spin structures on $\Sigma$ differ by an element of $H^1(\Sigma,\bz_2)$.

A hyperbolic spin structure is elegantly described in terms of the associated Fuchsian representation  
\[\overline{\varrho}:\pi_1\Sigma\to PSL(2,\br).\]  For $\overline{\Gamma}=\overline{\varrho}(\pi_1\Sigma)$,  
\[P_{SO}(\Sigma)=PSL(2,\br)/\overline{\Gamma}\to\bh/\overline{\Gamma}=\Sigma.
\]

\begin{exercise}
Prove that $P_{SO}(\bh)=PSL(2,\br)$.
\end{exercise}

We used $\overline{\varrho}$ above because we will instead consider
\[\varrho:\pi_1\Sigma\to SL(2,\br)\]
with image $\Gamma$
such that the composition $\overline{\varrho}$ of $\varrho$ with the map $SL(2,\br)\to PSL(2,\br)$ is Fuchsian. Any representation $\pi_1\Sigma\to PSL(2,\br)$ lifts (in many ways) to a representation $\pi_1\Sigma\to SL(2,\br)$ and the choice of lift defines a spin structure.  This can be seen from the following realisation of $P_{\text{Spin}}(\Sigma)$:
\[
P_{\text{Spin}}(\Sigma)=SL(2,\br)/\Gamma\to PSL(2,\br)/\overline{\Gamma}=P_{SO}(\Sigma).
\]
Associated to a hyperbolic spin structure is a natural representation
\[\varrho:\pi_1\Sigma\to SL(2,\br)\curvearrowright\br^2   
\] 
which defines a real rank two vector bundle 
\[T_\Sigma^{1/2}\to\Sigma.\]
When $\Sigma\cong C$ is compact, this bundle is related to the spin bundle $L$ satisfying $L^{\otimes 2}\cong \omega_C$ via
\begin{equation}   \label{flatspin}
T_\Sigma^{1/2}\otimes\bc\cong L\oplus L^{\vee}
\end{equation}
where $L^{\vee}$ denotes the dual of $L$.

An elegant description of the conformal construction of spin structures uses twisted curves or orbifolds, which are defined next.

\subsubsection{Twisted curves}

A \textit{twisted curve} with a group $ G $ is a one-dimensional orbifold, or stack, $ \cc $, satisfying the following conditions:
\begin{itemize}
    \item Generic points of $ \cc $ have trivial isotropy groups.
    \item Non-trivial orbifold points of $ \cc $ have isotropy groups isomorphic to $ G $.
\end{itemize}

A twisted curve is equipped with a morphism $ p \colon \cc \to C $, called the \textit{coarse map}, which forgets the orbifold structure. Here, $ C $ is a smooth curve and is referred to as the \textit{coarse curve} of $ \cc $, also denoted $C=|\cc|$.

A good source of examples of twisted curves arises from branched covers.  A branched cover of curves
\[
f:X\to C
\]
is a holomorphic map such that, near each ramification point $p\in X$, there exist local coordinates $z$ centred at $p$ and $w$ centred at $f(p)$ for which
\[
w=z^r,
\]
where $r\ge 2$ is the ramification index at $p$.   Although it is standard to consider the target $C$ to be a curve without orbifold structure, it naturally comes equipped with an orbifold structure, where the branch point $f(p)$ has isotropy group $\bz_r$.
\begin{exercise}
Prove that for any elliptic curve $E=\bc/\Lambda$ the quotient 
\[E/\sigma=(\bp^1,\{p_1,p_2,p_3,p_4\})\] 
for $\sigma(z)=-z$ with fixed points mapping to $p_i$, is a twisted curve (of orbifold Euler characteristic 0).
\end{exercise}

From now on, we only consider twisted curves with group $G=\bz_2$ such that the points $p_i\in\cc$ with non-trivial isotropy group $\bz_2$ are precisely the labeled points.   A local chart around each point $p_i\in\cc$ corresponds to $z\mapsto z^2$. 

\subsubsection{Orbifold line bundles} \label{orbdle} An orbifold line bundle (or simply a line bundle) $\cf$ over a smooth twisted curve $\cc$ is a locally equivariant bundle over the local charts around each point.

At each orbifold point $p$ an orbifold line bundle associates a representation of $\bz_2$ on $\cf|_p$ acting by multiplication by $\pm 1$  which can be seen as follows.  Locally, a $\bz_2$-orbifold point is described by the quotient of a disk $D=\{z:|z|<\epsilon\}$ by the $\bz_2$-action $z\mapsto-z$  which, via $z\mapsto z^2=x$,  maps  to the disk $\{x:|x|<\epsilon^2\}$ with $\bz_2$-orbifold point at $x=0$.  A locally equivariant line bundle is given by a lift of the $\bz_2$-action on $D$ to a $\bz_2$-action on $D\times\bc$.  There are two such lifts up to equivalence represented by  
\[(z,v)\mapsto(-z,v)\qquad\text{or}\qquad(z,v)\mapsto(-z,-v).\]  
Over the orbifold point, the representation is trivial in the first case and non-trivial in the second case.

An important example of an orbifold line bundle $\cf$ over $\cc$ is the canonical bundle $\omega_\cc$.  It is locally generated by $dx$ for any local coordinate $x$.  At an orbifold point $x=z^2$, the canonical bundle $\omega_\cc$ is generated by $dz$ hence the local representation is non-trivial, i.e. $dz\mapsto-dz$ under $z\mapsto -z$.  In particular, it has no square-root $\cf^{\otimes 2}\cong\omega_\cc$ since the square of any representation $\bz_2\to\bz_2$ is trivial.  Define the log canonical bundle of a twisted curve $\cc$ by
\[\omega_{\cc}^{\text{log}}:=\omega_{\cc}(D),\quad D=p_1+...+p_n.\]
The log-canonical bundle $\omega_\cc^{\text{log}}$ is generated by $\frac{dx}{x}=2\frac{dz}{z}\stackrel{z\mapsto-z}{\longmapsto} 2\frac{dz}{z}$ hence the local representation is trivial and a square-root may exist.
Over the coarse curve, $\omega_C$ is generated by $dx=2zdz$ so we see that $\rho^*\omega_C\not\cong\omega_\cc$.  However, $\omega_C\cong \rho_*\omega_\cc$.  Moreover, $\deg\omega_C=2g-2$ and 
$$\deg\omega_\cc=2g-2+\frac12n$$
where half-integer degree exists since non-trivial orbifold points have degree $\frac12$.  
For $\omega_\cc^{\text{log}}=\omega_\cc(p_1+...+p_n)$, since $\frac{dx}{x}=2\frac{dz}{z}$ then $\rho^*\omega_C^{\text{log}}\cong\omega_\cc^{\text{log}}$ and 
\[\deg\omega_C^{\text{log}}=2g-2+n=\deg\omega_\cc^{\text{log}}.\]
\begin{definition}
A spin structure on a twisted pointed curve $\cc$ with isotropy group $\bz_2$ is given by $(\cc,p_1,\ldots,p_n,\cl,\phi)$ where $\cl\to\cc$ is a line bundle equipped with an isomorphism $\phi:\cl^{\otimes 2}\xrightarrow{\cong}\omega_C^{\log}$.
\end{definition}
The spin bundle $\cl$ is an orbifold line bundle over $\cc$ of degree
\[\deg \cl=g-1+\frac12n\]  
which may be a half-integer.

 \subsection{Neveu--Schwarz and Ramond boundary behaviour}   \label{boundbeh}

At a labeled point, or boundary component, $p_j\in\cc$, the behaviour of a spin structure is of one of two types, known as Neveu--Schwarz and Ramond..  This corresponds to the two possible representations $\bz_2\to\bz_2$, non-trivial or trivial, at  a labeled point in the twisted curve construction.  It also corresponds to the two lifts of a peripheral element from $PSL(2,\br)$ to $SL(2,\br)$ in the hyperbolic construction.
\begin{definition} On a twisted curve $\cc$, given a spin structure defined by $\cl^2\cong\omega_\cc^{\log}$ and its uniformising representation $\varrho:\pi_1\Sigma\to SL(2,\br)$, {\em Neveu--Schwarz} behaviour of the spin structure at $p_j\in\cc$, or equivalently $\beta_j\subset\partial\Sigma\cong\cc-\{p_1,...,p_n\}$, is defined by the following equivalent conditions:
\begin{itemize}
\item the representation $\bz_2\to\bz_2$ induced by $\cl$ at $p_j$ is non-trivial;
\item $\tr(\varrho(\beta_j))<0$;
\item the spin structure on $\Sigma$ extends across $p_j$.
\end{itemize}
It is {\em Ramond} if it is not Neveu--Schwarz.
\end{definition}
The relationship between these three conditions can be seen by considering a spin structure on a disk: the tangent frame induced from the boundary circle $S^1\hookrightarrow P_{SO}(\Sigma)$ rotates exactly once relative to a trivialisation over the disk, so the square root rotates by half and defines holonomy of multiplication by $-1$.  Equivalently the element of $H^1(P_{SO}(\Sigma),\bz_2)$ sends it to the non-trivial element of $\bz_2$.  The following exercise shows that the boundary behaviour of a spin structure on a once punctured torus is necessarily of Neveu--Schwarz type, in particular it is independent of the lift of the Fuchsian representation from $PSL(2,\br)$ to $SL(2,\br)$.
\begin{exercise}
For any Fuchsian representation $\varrho:\pi_1(\Sigma)\to SL(2,\br)$ defining a hyperbolic structure on a once punctured torus $\Sigma$, prove that $\tr\varrho(\beta_1)<0$. 
\end{exercise}

\subsubsection{The number of Ramond points is  even.} \label{numram} On any surface with spin structure, the number of Ramond points is always even.  This can be seen from both the conformal and hyperbolic viewpoints.  We will begin with the conformal viewpoint via twisted curves.

Denote the divisor of Neveu--Schwarz points by $NS$ of degree $|NS|$ and the divisor of Ramond points by $R$ of degree $|R|$.  
For $\cl^2\cong\omega_{\cc}^{\text{log}}$, we have 
\[\deg \cl=g-1+\frac12n.\]   Local Neveu--Schwarz sections of the pushforward sheaf are skew-invariant under $z\to-z$ which forces them to vanish.  This decreases the degree by a half at each Neveu--Schwarz point while the degree at each Ramond point is unchanged.  Hence for $L=\rho_*\cl$ (or more precisely the sheaf of local sections pushes forward to the sheaf $\co_C(L)=\rho_*\co_{\cc}(\cl)$) we have
\[\deg L=g-1+\frac12n-\frac12|NS|=g-1+\frac12|R|.\]
Since the degree of any divisor on the coarse curve must be an integer this shows that the number $|R|$ of Ramond boundary components is even.   This evenness argument often appears entirely over the coarse curve, and we give it here to show that it is equivalent.   The skew-invariance argument above shows that a spin bundle, satisfying $\cl^2\cong\omega_{\cc}^{\text{log}}$ over $\cc$, pushes forward to a bundle $L=\rho_*\cl$ on the coarse curve $C$ that satisfies
\begin{equation}   \label{coarspin}
L^{\otimes 2}\cong \omega_C^{\text{log}}(-NS)=\omega_C(R)
\end{equation}
then taking degrees gives
\[
2\deg L=\deg \omega_C+|R|=(2g-2)+\deg R.
\]
hence the number of Ramond points $|R|$ is even.

Evenness of $|R|$ can also be proven via the following homological argument. 
A spin structure on a surface $C$ determines a quadratic form
\[
q:H_1(C,\mathbb Z_2)\to\mathbb Z_2
\]
satisfying
\begin{equation}  \label{quadadd}
q(a+b)=q(a)+q(b)+a\cdot b,
\end{equation}
where $a\cdot b$ denotes the mod $2$ intersection number.

The quadratic form naturally associated to a spin structure, due to Johnson \cite{JohSpi}, is defined as follows. Let
\[
[C]\in H_1(\Sigma,\mathbb Z_2)
\]
be represented by a finite collection of pairwise disjoint embedded oriented closed curves
\[
C=C_1\sqcup\cdots\sqcup C_n.
\]
Define a map
\[
\ell:H_1(\Sigma,\mathbb Z_2)\longrightarrow
H_1(P_{SO}(\Sigma),\mathbb Z_2)
\]
by
\[
\ell([C])
=
n\,\sigma+\sum_{i=1}^n \widetilde C_i,
\]
where $\sigma$ is the image of the generator of
\[
H_1(SO(2),\mathbb Z_2)
\]
under the inclusion of a fibre
\[
SO(2)\hookrightarrow P_{SO}(\Sigma),
\]
and $\widetilde C_i$ is the lift of $C_i$ to the orthonormal frame bundle using its tangential framing.

The map $\ell$ is well defined on homology. Indeed, it is invariant under isotopy of the representative curves, vanishes on the boundary of a disk (whose tangential lift represents $\sigma$), and is unchanged under the local surgery that replaces a crossing by embedded curves. Consequently, $\ell([C])$ depends only on the homology class $[C]$.
As described in \eqref{spinlift}, this defines a spin structure via a cohomology class
$
\eta\in H^1(P_{SO}(\Sigma),\mathbb Z_2)
$
satisfying $\eta(\sigma)=1.$

The associated quadratic form is then defined by
\[
q_\eta:H_1(\Sigma,\mathbb Z_2)\to\mathbb Z_2,
\qquad
q_\eta=\eta\circ\ell.
\]
It is straightforward to verify that $q_\eta$ is a quadratic form:
\[
q_\eta(a+b)
=
q_\eta(a)+q_\eta(b)+a\cdot b,\qquad
\forall
a,b\in H_1(\Sigma,\mathbb Z_2).
\]
Moreover, the correspondence
\[
\eta\longmapsto q_\eta
\]
is an isomorphism of affine $H^1(\Sigma,\mathbb Z_2)$-spaces between the set of spin structures on $\Sigma$ and the set of quadratic forms.

Neveu--Schwarz and Ramond boundary classes of a spin structure are natural in terms of the quadratic form of a spin structure.  
This follows from considering spin structures on the boundary of a disk $D$.  Equip $D$ with its unique spin structure. Its restriction to the boundary is of Neveu--Schwarz  type. Since $D$ is contractible, its orthonormal frame bundle is trivial:
\[
P_{SO}(D)\cong D\times SO(2).
\]
The tangential framing of the boundary circle $\partial D$ has winding number $1$ relative to this trivialisation. Consequently, its tangential lift
\[
\widetilde{\partial D}\subset P_{SO}(D)
\]
represents the nontrivial class in the fibre direction, and hence the spin structure $\eta$ satisfies
\[
\eta(\widetilde{\partial D})=1.
\]
By definition of Johnson's map,
\[
\ell([\partial D])=\sigma+\widetilde{\partial D},
\]
where $\sigma$ denotes the generator of
\[
H_1(SO(2),\mathbb Z_2).
\]
Since $\eta(\sigma)=1$, it follows that
\[
q([\partial D])
=
\eta\!\bigl(\ell([\partial D])\bigr)
=
\eta(\sigma+\widetilde{\partial D})
=
\eta(\sigma)+\eta(\widetilde{\partial D})
=
1+1
=
0
\]
in $\mathbb Z_2$.
Thus the quadratic form of a spin structure that extends to a disk vanishes on the boundary of a disk:
\[
q([\partial D])=0
\]
so this holds precisely for Neveu--Schwarz boundary components
and we conclude that for any boundary class $[\gamma]\in H_1(\Sigma)$:
\begin{align*}
q([\gamma])=0&\quad\Leftrightarrow\quad \gamma\text{ is Neveu--Schwarz}\\
q([\gamma])=1&\quad\Leftrightarrow\quad \gamma\text{ is Ramond}.
\end{align*}

Let $\gamma_i$, $i=1,...,n$ be the boundary classes represented by small positively oriented loops around the boundary components $\beta_i$ of $\Sigma$. Since the union of the boundary circles bounds a surface, their homology classes satisfy
\[
[\gamma_1]+\cdots+[\gamma_n]=0
\]
in $H_1(\Sigma,\mathbb Z_2)$.
Applying the quadratic form to this relation yields
\[
0=q\bigl([\gamma_1]+\cdots+[\gamma_n]\bigr).
\]
Since the loops $\gamma_i$ are pairwise disjoint, all intersection numbers vanish:
\[
[\gamma_i]\cdot[\gamma_j]=0,
\qquad i\neq j.
\]
Using the additive property \eqref{quadadd} repeatedly gives
\[
q\bigl([\gamma_1]+\cdots+[\gamma_n]\bigr)
=
q([\gamma_1])+\cdots+q([\gamma_n]).
\]
Hence
\[
0
=
q([\gamma_1])+\cdots+q([\gamma_n])
=
1+\cdots+1+0+\cdots+0
=
|R|
\pmod 2.
\]
Thus the number of Ramond points of a spin structure is necessarily even.
 
\subsection{Moduli space ${\cal M}_{g,n}^{\text{spin}}$}

In this section we define the moduli space ${\cal M}_{g,n}^{\text{spin}}$ for $2g-2+n>0$, which parametrises genus $g$, $n$-pointed, spin, Riemann surfaces up to conformal equivalence.  We begin with different equivalent descriptions of the moduli space of Riemann surfaces ${\cal M}_{g,n}$, one of the central objects in geometry.   We define ${\cal M}_{g,n}$ in different 
ways, where in each case roughly speaking, a point of
${\cal M}_{g,n}$ parametrises an equivalence class of a surface together with
additional geometric data.

\begin{definition}
A \emph{genus $g$, $n$-pointed Riemann surface} is a pair
\[
(C;p_1,\ldots,p_n),
\]
\begin{itemize}
\item $C$ is a compact connected Riemann surface of genus $g$,
\item $p_1,\ldots,p_n\in C$ are distinct labeled points.
\end{itemize}
Two such objects
$(C;p_1,\ldots,p_n)$ and $(C';p_1',\ldots,p_n')$
are equivalent if there exists a biholomorphic map
$f:C\to C'$ satisfying
$f(p_i)=p_i'$ for all $i$.
\end{definition}

The moduli space ${\cal M}_{g,n}$ is the set of equivalence classes of
such $n$-pointed Riemann surfaces.  
\[{\cal M}_{g,n}=\{(C,p_1,...,p_n)\mid \text{genus }g,\ n\text{-pointed Riemann surface }C\}/\sim.\]
When
\[
2g-2+n>0,
\]
the topological surface $\Sigma\cong C-\{p_1,\ldots,p_n\}$ admits a unique complete hyperbolic metric in the conformal
class of $C$, with a cusp at each labeled point.  Consequently, one may
equivalently regard ${\cal M}_{g,n}$ as the moduli space of complete hyperbolic surfaces of genus $g$ with $n$ labeled cusps up to isometry.  Note that a neighbourhood of a cusp is also known as a boundary component and denoted $\beta_i\subset\partial \Sigma$.
\[
{\cal M}_{g,n}=\Big\{(\Sigma,\beta_1,...,\beta_n)\mid \Sigma \text{ complete, oriented genus }g\text{ hyp.} \text{ surf., }\\ \partial \Sigma=\sqcup\beta_i\Big\}/\sim
\]
The equivalence of these two descriptions of ${\cal M}_{g,n}$ is a consequence of the uniformisation theorem.

\vspace{0.2cm}

To describe the geometry of ${\cal M}_{g,n}$ it is useful to introduce
\emph{Teichm\"uller space}.  Let $\Sigma$ be a fixed oriented
topological surface of genus $g$ with $n$ labeled points $x_1,...,x_n$.

\begin{definition}
A \emph{marking} of an $n$-pointed Riemann surface $C$ is an orientation-preserving
homeomorphism
\[
f:\Sigma\to C
\]
that sends labeled points to  labeled points: $f(x_i)=p_i$.
Two marked surfaces $(C,f)$ and $(C',f')$ are equivalent if there
exists a biholomorphism $h:C\to C'$ such that
$f'$ is isotopic to $h\circ f$.
\end{definition}

The set of equivalence classes of marked surfaces is the
\emph{Teichm\"uller space} $T_{g,n}$. 
Unlike ${\cal M}_{g,n}$, which is generally an orbifold,
$T_{g,n}$ is a smooth contractible manifold of real dimension
\[
\dim T_{g,n}=6g-6+2n.
\]
The mapping class group
\[
\Gamma_{g,n}=\pi_0(\mathrm{Homeo}^+(\Sigma,x_1,...,x_n))
\]
acts on $T_{g,n}$ by changing the marking, and the moduli space is the
quotient
\[
{\cal M}_{g,n}=T_{g,n}/\Gamma_{g,n}.
\]

\vspace{0.2cm}

A third description of the moduli space via its universal cover uses representations of the fundamental group.
Given a complete hyperbolic structure on $\Sigma$, its universal cover
is the hyperbolic plane $\mathbb H^2$, and the deck transformations
define a representation
\[
\varrho:\pi_1(\Sigma)\to PSL(2,\mathbb R),
\]
called the \emph{holonomy representation}.
This leads to the character variety
\[
\mathcal X_{g,n}
=
\mathrm{Hom}\!\left(\pi_1(\Sigma),PSL(2,\mathbb R)\right)
/PSL(2,\mathbb R),
\]
whose points are conjugacy classes of representations.

The representations arising from hyperbolic structures, known as {\em Fuchsian}, are precisely
the discrete and faithful representations whose boundary classes, given by peripheral loops around
the punctures, are parabolic.  Note that any closed curve $\gamma\subset \Sigma$ corresponds to a conjugacy class $[\gamma]\subset\pi_1\Sigma$.  A conjugacy class in $PSL(2,\br)$ is {\em parabolic} if any representative $A\in PSL(2,\br)$ satisfies $|\tr(A)|=2$ and {\em hyperbolic} if any representative $A\in PSL(2,\br)$ satisfies $|\tr(A)|>2$.   Thus  every non-trivial element in the image of a Fuchsian representation is represented by a hyperbolic or parabolic isometry of
$\mathbb H^2$, so that
\[
|\operatorname{tr}(\varrho(\gamma))|\ge 2
\qquad
\text{for all }\gamma\in\pi_1(\Sigma).
\]
 The corresponding connected component of the character variety is the
\emph{Teichm\"uller component}.  The holonomy map identifies
Teichm\"uller space with this component,
\[
T_{g,n}
\cong
\mathcal X^{\mathrm{Teich}}_{g,n}.
\]
Thus, Teichm\"uller space may be viewed either as the space of
\begin{itemize}
\item marked Riemann surfaces,
\item marked complete hyperbolic structures,
\item discrete faithful $PSL(2,\mathbb R)$ representations of
$\pi_1(\Sigma)$.
\end{itemize}
The forgetful map that discards the marking gives the orbifold covering
\[
T_{g,n}\longrightarrow {\cal M}_{g,n}.
\] 
 
\subsubsection{Moduli space of spin curves} 
Given a compact complex curve or equivalently a compact Riemann surface $C\cong\Sigma$, define the moduli space of spin curves and the moduli space of spin hyperbolic surfaces by:
\begin{align*}
{\cal M}_{g}^{\rm spin}=&\{(C,L,\phi)\mid g(C)=g,\phi:L^2\stackrel{\cong}{\longrightarrow}\omega_{C}\}/\sim\\
\cong
&\Big\{(\Sigma,T_\Sigma^{1/2})\mid\Sigma\text{ oriented hyperbolic},\ g(\Sigma)=g\Big\}/\sim
\end{align*}
where the first quotient is by conformal maps and the second by isometries, and $P_{\text{Spin}}(\Sigma)$ is the principal bundle of the rank two vector bundle $T_\Sigma^{1/2}$. The relationship between the two descriptions is obtained by taking the unique hyperbolic metric on $\Sigma$ in the conformal class of $C$ and using \eqref{flatspin} to identify spin structures.

Now introduce labeled points, or boundary components.  The definition of the moduli space of spin hyperbolic surfaces is almost unchanged from the $n=0$ case, with the added condition of completeness of the metric:
\begin{align*}
{\cal M}_{g,n}^{\rm spin}=\Big\{(\Sigma, T_\Sigma^{1/2},\beta_1,...,\beta_n)\mid \Sigma \text{ complete, oriented}&  \text{ hyperbolic surface,}\\
& g(\Sigma)=g,\ \partial \Sigma=\sqcup\beta_i\Big\}/\hspace{-1mm}\sim
\end{align*}
The spin Teichm\"uller space $\ct_{g,n}^{\rm spin}$ is given by the Fuchsian component of the spin character variety
\begin{align*}
\mathcal X_{g,n}^{\rm spin}=\Big\{\varrho:\pi_1(\Sigma)\to SL(2,\br)\mid |\text{tr}(\varrho(\gamma))|&\geq 2,\ \forall\gamma\in\pi_1(\Sigma),
\\
 &g(\Sigma)=g,\ \partial \Sigma=\sqcup\beta_i\Big\}/\hspace{-1mm}\sim
\end{align*}
where a representation is Fuchsian if its projection to $PSL(2,\br)$ is discrete and faithful, and the quotient is by conjugation.  

The corresponding construction of ${\cal M}_{g,n}^{\rm spin}$ for $n>0$ using the conformal construction of spin structures is most elegantly described in terms of twisted curves.   
\begin{definition}
The moduli space of spin curves is
\[
{\cal M}_{g,n}^{\mathrm{spin}}
=
\left\{
(\cc,p_1,\ldots,p_n,\cl,\phi)
\;\middle|\;
\phi:\cl^{\otimes 2}\xrightarrow{\cong}\omega_\cc^{\log}
\right\}/\!\sim
\]
where each $\cc$ is a twisted curve with isotropy group $\bz_2$.
\end{definition}
The formulation of the moduli space of spin curves over the coarse curve $C=|\cc|$ replaces $\omega_\cc^{\log}$ with twists of $\omega_C$, given in \eqref{coarspin},  by divisors determined by Neveu--Schwarz and Ramond behaviour at the marked points and nodes. 

Denote the moduli space of twisted curves with $\mathbb{Z}_{2}$ isotropy by ${\cal M}^{(2)}_{g,n}$.  The natural forgetful map ${\cal M}_{g,n}^{\mathrm{spin}}\to{\cal M}^{(2)}_{g,n}\to{\cal M}_{g,n}$ that forgets the spin structure and twisted curve structure is denoted by
\[ p:{\cal M}_{g,n}^{\mathrm{spin}}\to{\cal M}_{g,n}.
\]

The spin Teichm\"uller space is no longer connected, due to different boundary behaviour of the spin structure described in Section~\ref{boundbeh}.  Each connected component is naturally homeomorphic to the usual Teichm\"uller space $\ct_{g,n}$ via the natural quotient map $SL(2,\br)\to PSL(2,\br)$ which forgets the spin structure.

\subsubsection{Components of the moduli space of spin curves}  \label{spincomp}
The moduli space of twisted spin curves decomposes into components determined by the boundary behaviour:
\[
{\cal M}_{g,n}^{\rm spin}=\bigsqcup_{\sigma\in\{0,1\}^n}{\cal M}_{g,\sigma}^{\rm spin}
\]
where $\sigma_j=0$ for Neveu--Schwarz $p_j$ and $\sigma_j=1$ for Ramond $p_j$. The disjoint union is over all $\sigma$ satisfying $|\sigma|$ is even since the number of Ramond points is even.  Each component ${\cal M}_{g,\sigma}^{\rm spin}$ is connected except for $|\sigma|=n$ when it is the union of two components.  

When all boundary classes of a spin structure are Neveu--Schwarz it is the pull-back under
$\rho:\cc\longrightarrow C$ of a spin structure on the coarse curve $C=|\cc|$.
Equivalently, for $\Sigma=C\setminus\{p_1,\ldots,p_n\}$, 
the spin structure extends to $C$. 
Consequently, the associated quadratic form
\[
q:H_1(\Sigma,\bz_2)\longrightarrow\bz_2
\]
is the pull-back of a quadratic form on the symplectic vector space $H_1(C,\bz_2)$.  
The {\em Arf invariant} of a quadratic form $q$ defined on a symplectic vector space over $\bz_2$ is a $\bz_2$-valued invariant defined by
\[\text{Arf}(q)=\sum_{i=1}^g q(\alpha_i)q(\beta_i)\]
for any standard symplectic basis $\{\alpha_1,\beta_1,...,\alpha_g,\beta_g\}$ 
meaning that $(\alpha_i,\beta_j)=\delta_{ij}$, $(\alpha_i,\alpha_j)=0=(\beta_i,\beta_j)$.  This is independent of the choice of $\{\alpha_i,\beta_i\}$.  A spin structure is {\em even} if its quadratic form has even Arf invariant and {\em odd} if its quadratic form has odd Arf invariant.   This gives rise to two connected components determined by their Arf invariant,  when $|\sigma|=n$.

\section{Sheaf cohomology over spin surfaces}  \label{sheafc}
The moduli space of spin curves comes equipped with a natural vector bundle and a natural Euler class.  This gives rise to a natural differential form on the moduli space, and a cohomology class after compactification, which pairs with the Weil--Petersson form to produce interesting new measures on the moduli space.  Measures arise here as top dimensional differential forms.  In this section we define the bundle in two equivalent ways, via the conformal viewpoint and the hyperbolic viewpoint.

\subsection{Local systems and holomorphic spin bundles.}   \label{sheafcoh}

For $\Sigma\cong\cc-D$, the flat structure on $T_\Sigma^\frac12$ defines a holomorphic structure on a rank two bundle over $\cc$.  When $D=\emptyset$, so that $\Sigma$ is compact, there is an isomorphism of holomorphic vector bundles 
\[T_\Sigma^\frac12\otimes\bc\cong \cl\oplus \cl^\vee\]
for $\cl^{\otimes 2}\cong\omega_\cc$.
The retraction of $SL(2,\br)$ to $SO(2)\subset SL(2,\br)$ allows the reduction of the structure group of any $SL(2,\br)$ bundle $E$ to an $SO(2)$ bundle $F$, thus a holomorphic line bundle, and there is a natural isomorphism $E\otimes\bc\cong F\oplus F^\vee$.  Using this together with the decomposition of a holomorphic vector bundle into unique decomposables \cite{AtiKru} produces the isomorphism of real rank two bundles:
\[T_\Sigma^\frac12\cong \cl^\vee.\] 
For general $D$, there is a canonical extension of $T_\Sigma^\frac12\otimes\bc$ from a bundle over $\Sigma$ to a rank two holomorphic vector bundle $E_\Sigma\to\cc$.  It is obtained from a generalisation of the extension over smooth curves due to Deligne \cite{DelEqu}.  Again there is an isomorphism of holomorphic vector bundles $E_\Sigma\cong \cl\oplus \cl^\vee$. 
One can prove this in two ways, by showing that Deligne's local extension over a disk above an orbifold point is equivariant hence applies to twisted curves.   Or, by applying Deligne's canonical extension over the coarse curve $p:\cc\to C$ to produce the bundle $L\oplus L^\vee(-NS)$ over $C$, for $L=C_*\cl$, which is the pushforward of $\cl\oplus \cl^\vee$.  Recall that $NS$ is the divisor of Neveu-Schwartz points and $\rho_*(\cl^\vee)=(\rho_*\cl)^\vee(-NS)$.

The sheaf of locally constant sections of $T_\Sigma^\frac12$ defines sheaf cohomology groups which we denote 
$H^j_{dR}(\Sigma,T_\Sigma^\frac12)$, for $j=0,1,2$ since they can be calculated via a twisted de Rham complex using the flat connection.  

The sheaf of locally holomorphic sections of $ \cl^\vee$ defines sheaf cohomology groups $H^j(\cc, \cl^{\vee})$ for $j=0,1$, and $H^0(\cc, \cl^{\vee})=0$, the space of global holomorphic sections is trivial, since $\deg \cl^{\vee}<0$.  We will see that they have similar behaviour to the sheaf cohomology groups $H^j_{dR}(\Sigma,T_\Sigma^\frac12)$.

An element of $H^0_{dR}(\Sigma,T_\Sigma^\frac12)$ corresponds to a horizontal section or equivalently an invariant vector of the Fuchsian representation $\varrho$, which is an eigenvector with eigenvalue one, $\varrho(\gamma)v=v$ for all $\gamma\in\pi_1(\Sigma)$.  But $|\tr(\varrho(\gamma))|>2$ for some $\gamma\in\pi_1(\Sigma)$ which implies $v=0$, hence 
\[H^0_{dR}(\Sigma,T_\Sigma^\frac12)=0.\]  
Also, 
\[H^2_{dR}(\Sigma,T_\Sigma^\frac12)=0\] 
which uses a retraction of $\Sigma$ to a one-dimensional spine
when $\Sigma$ is non-compact, hence any $H^2$ vanishes.  When $\Sigma$ is compact it uses Poincar\'e duality
\[H^2_{dR}(\Sigma,T_\Sigma^\frac12)\cong H_0(\Sigma,T_\Sigma^\frac12)\cong H^0_{dR}(\Sigma,(T_\Sigma^\frac12)^\vee)^\vee= 0.\] 
The final equality uses the fact that any non-zero invariant vector in the dual vector space gives a non-trivial invariant one-dimensional subspace of the representation $\varrho$, contradicting irreducilbility.

If all boundary points are Neveu--Schwarz, then there is a canonical isomorphism
\begin{equation}  \label{canis}
H^1_{dR}(\Sigma,T_\Sigma^\frac12)\cong H^1(\cc, \cl^\vee)^\vee
\end{equation} 
proven in \cite{NorEnu} using a theorem of Simpson applied to the (complex) rank two bundle $L\oplus L^\vee$ equipped with a natural Higgs field for $L=\rho_*\cl$.

When there are Ramond points, the isomorphism \eqref{canis} no longer holds.  It is instead replaced by an exact sequence.  To prove this generalisation, we compute the sheaf cohomology
$H^k_{\mathrm{dR}}(\Sigma,T_\Sigma^{1/2})$.  
Following \cite{NorEnu,SWiJTG}, we instead compute the twisted homology of a dual cellular complex 
\[
H_k(\Sigma,T_\Sigma^{1/2})^\vee
\cong
H^k_{\mathrm{dR}}(\Sigma,T_\Sigma^{1/2})
\]
where the isomorphism uses the symplectic structure on $T_\Sigma^{1/2}$.  The twisted homology groups $H_k(\Sigma,T_\Sigma^{1/2})$ can be defined as follows.  Given an ideal triangulation $\ct$ of $\Sigma$, for each cell $\sigma\in\ct$, let
\[
\cv_\sigma
=
H^0(\sigma,T_\Sigma^{1/2})
\]
denote the space of covariantly constant sections of $T_\Sigma^{1/2}$ pulled back over  $\sigma$.  Define the cellular chain groups by
\[
C_k(\Sigma,T_\Sigma^{1/2})
=
\bigoplus_{\sigma\in\ct_k}\cv_\sigma.
\]
The boundary maps are defined by
\[
\begin{array}{rcl}
C_{k+1}(\Sigma,T_\Sigma^{1/2})
&\stackrel{\partial}{\longrightarrow}&
C_k(\Sigma,T_\Sigma^{1/2})\\[0.3cm]
s|_\sigma
&\longmapsto&
s|_{\partial\sigma}
=
\displaystyle\bigoplus_i
(-1)^{\epsilon_i}s|_{\sigma_i},
\end{array}
\]
where
\[
\partial\sigma
=
\bigcup_i
(-1)^{\epsilon_i}\sigma_i
\]
is the usual oriented boundary decomposition into codimension-one cells.


The identity
\[
\partial^2=0
\]
follows from the usual cancellation of codimension-two faces. For a $2$-cell, the contribution at each vertex consists of the same covariantly constant section, transported to the vertex along the two incident edges, appearing with opposite signs because of the orientations. Equivalently, this is the vanishing of the square of the ordinary simplicial boundary map. The same argument applies in higher dimensions, where each codimension-two cell appears twice with opposite orientations. Since the connection is flat, parallel transport around the boundary of a sufficiently small cell is trivial, so the two extensions agree.

We define $H_k(\Sigma,T_\Sigma^{1/2})$ to be the homology of the resulting chain complex
\begin{equation}  \label{twicom}
C_2(\Sigma,T_\Sigma^{1/2})
\stackrel{\partial}{\longrightarrow}
C_1(\Sigma,T_\Sigma^{1/2})
\stackrel{\partial}{\longrightarrow}
C_0(\Sigma,T_\Sigma^{1/2}).
\end{equation}

The following proposition uses the notion of a graph connection which we define here.
\begin{definition}
A \emph{ribbon graph} is a graph equipped with a cyclic ordering of the half-edges incident at each vertex.
Equivalently, it is an embedding
\[
\Gamma\hookrightarrow S
\]
of a graph into an oriented surface $S$ such that each connected component of $
S\setminus\Gamma$
is an open disk. The orientation of $S$ induces a cyclic ordering of the half-edges incident at every vertex.

Combinatorially, a ribbon graph is a triple
\[
(X,\sigma_0,\sigma_1),
\]
where $X$ is the set of oriented edges, $\sigma_1:X\to X$ is a fixed-point-free involution reversing the orientation of each edge, and $\sigma_0:X\to X$ is a permutation whose cycles give the cyclic ordering of oriented half-edges around each vertex.
The vertices of the ribbon graph are the orbits
\[
V=X/\sigma_0,
\]
and the edges are the orbits
\[
E=X/\sigma_1.
\]
The boundary components (or faces) are the orbits of the permutation
\[
\sigma_2=\sigma_0^{-1}\sigma_1.
\]
\end{definition}
\noindent To an oriented edge $e\in X$, we denote its source and target vertices by 
\[e_-:=\pi_0(e),\qquad e_+:=\pi_0\circ\tau_1(e)\] 
where $\pi_0:X\to X/\tau_0$ is the quotient map.
\begin{definition}
A \emph{graph connection} with group $G$ on a ribbon graph $(X,\sigma_0,\sigma_1)$  is a map
\[
g:X\longrightarrow G,\qquad g\circ\sigma_1=g^{-1}.
\]
Two graph connections are equivalent under \emph{gauge transformations} 
\[
h:X\longrightarrow G,\qquad h\circ\sigma_0=h
\]
which act by 
\[
g\longmapsto h\circ\sigma_1\cdot g\cdot h^{-1}.
\]
\end{definition}

\begin{thm}   \label{exactseq}
Given a spin structure $\cl^2\cong\omega_\cc^{\log}$ with Ramond points $R\subset\cc$ and a representation $\varrho:\pi_1(\Sigma)\to SL(2,\br)$ which defines a conformally equivalent complete hyperbolic spin structure and the bundle $T_\Sigma^\frac12\to\Sigma=\cc-\{p_1,...,p_n\}$, there is an exact sequence of real vector spaces
\begin{equation}   \label{ramex}
0\to\br^R\to H^1_{dR}(\Sigma,T_\Sigma^\frac12)\to H^1(\cc, \cl^\vee)^\vee\to0.
\end{equation}
\end{thm}
\begin{proof}
The proof is divided into six parts.  It generalises the proof of the Neveu--Schwarz case \eqref{canis} from \cite{NorEnu} and necessarily deviates from that proof in parts 3, 5 and 6.
\begin{enumerate}
\item[1.] $H^1_{dR}(\Sigma,T_\Sigma^\frac12)$ in terms of graph connections.
\item[2.] Holonomy with no non-trivial fixed points.
\item[3.] Holonomy with non-trivial fixed points.
\item[4.] Calculation of $H^1(\cc,L^\vee)$.
\item[5.] Construction of Ramond boundary solutions.
\item[6.] Linear independence of Ramond boundary solutions.
\end{enumerate}
\noindent {\em 1. $H^1_{dR}(\Sigma,T_\Sigma^\frac12)$ in terms of graph connections.} If $R=\emptyset$ then \eqref{ramex} follows from \eqref{canis} proven in \cite{NorEnu} so we assume that $R\neq\emptyset$.  Choose an ideal triangulation of $\Sigma$, which is necessarily non-compact. Then only  $C_1$ and $C_2$, the 1-chains and 2-chains, appear in the complex \eqref{twicom} defining
$H_k(\Sigma,T_\Sigma^{1/2})$.  The dual triangulation is a trivalent ribbon graph $\Gamma$.  We choose the ideal triangulation so that 
that there exists a collection of $2g-2+n$ edges $\cd\subset E(\Gamma)$ such that
\begin{equation}   \label{cohdim}
H^1_{dR}(\Sigma,T_\Sigma^\frac12)\cong(\br^2)^\cd.
\end{equation}
The collection $\cd$ will be a dimer with the property that each vertex of $\Gamma$ is incident to a unique edge in $\cd$.  More generally $\cd$ can be related to a dimer via Whitehead moves on the graph, but we will only consider  ideal triangulations with dual ribbon graphs admitting dimers.  The isomorphism \eqref{cohdim} is proven in \cite{NorEnu} in the case that all boundary components are of Neveu--Schwarz type.  
We recall the argument here, in order to modify it in the presence of Ramond punctures.  The representation $\varrho:\pi_1(\Sigma)\to SL(2,\br)$ up to conjugation is equivalent to a graph connection up to gauge transformation.  It associates an element $g_e$ to each oriented edge of $\Gamma$ satisfying $g_{\bar{e}}=g_e^{-1}$.   The vector space $H^1_{dR}(\Sigma,T_\Sigma^\frac12)$ is isomorphic to the dual of the cokernel of $C_2(\Sigma,T_\Sigma^{1/2})\stackrel{\partial}{\longrightarrow}C_1(\Sigma,T_\Sigma^{1/2})$ from \eqref{twicom} since $C^0=0$.  The dual of the cokernel is isomorphic to the kernel of the dual map 
\[\delta:(\br^2)^X\to(\br^2)^V
\] 
where $X$ is the set of oriented edges of $\Gamma$,
defined on $e\in X$ by
\[\delta v_e|_{e_+}=g_ev_e,\quad\delta v_e|_{e_-}=-v_e.\]
We choose the convention that the trivialisation of $T_\Sigma^\frac12$ over an oriented edge $e$ is induced from the trivialisation of $T_\Sigma^\frac12$ over its source vertex $e_-$.  
The condition $\delta v=0$ at each vertex is the vanishing of the sum of contributions from the three oriented edges adjacent to the given vertex, such as $\sum g_ev_e=0$ for a vertex with only incoming edges, or more generally each summand is $g_ev_e$ or $-v_e$.

Fix a dimer edge $e_0\in \cd$, choose an arbitrary non-zero vector 
\[ v_{e_0}\in\mathbb R^2, \] 
and set 
\[ v_e=0, \qquad e\in \cd\setminus\{e_0\}. \] 
Since $\Gamma$ is trivalent, the complement $\Gamma\setminus \cd$ is a disjoint union of embedded loops. Along an oriented loop 
\[ \gamma\subset\Gamma\setminus \cd, \] 
the vertex condition defining $\ker\delta$ uniquely determines the vector on each successive edge from the preceding one. For example, if the orientation of every edge agrees with that of $\gamma$, then consecutive edges satisfy \[ g_{e_i}v_{e_i}=v_{e_{i+1}}, \] where $e_i$ and $e_{i+1}$ are consecutive oriented edges of $\gamma$. Suppose first that $\gamma$ does not meet $e_0$. Then every vector on $\gamma$ is determined by a single vector $v_e$ on one edge, which must satisfy \[ g_\gamma v_e=v_e. \] 
Equivalently, $(g_\gamma-I)v_e=0$.\\

\noindent {\em 2. Holonomy with no non-trivial fixed points.}  First assume that $g_\gamma-I$ is invertible, which is the case treated in \cite{NorEnu}.   Invertibility of $g_\gamma-I$ holds when $g_\gamma$ is hyperbolic---always true on non-boundary $\gamma$---or when $g_\gamma$ is parabolic and Neveu--Schwarz, since $\tr(g_\gamma)<0$ and $\det(g_\gamma)=1$ $\Rightarrow$ 1 is not an eigenvalue of $g_\gamma$.  Invertibility of $g_\gamma-I$ implies that the only solution is $v_e=0$.  

Now denote components of $\Gamma\setminus \cd$ that meet the distinguished edge $e_0$ by $\gamma_1$ and $\gamma_2$ with labels $(g_i,v_i)\in(S L(2,\br),\br^2)$ for $i=1,2$ as in the diagrams below.  The diagrams show the three possibilities.  A dimer edge $e_0$ cannot be a loop so via a gauge transformation we may arrange $g_{e_0}=I$. 
\begin{figure}[H]
\label{compdumb}
\centering
\includegraphics[scale=0.3]{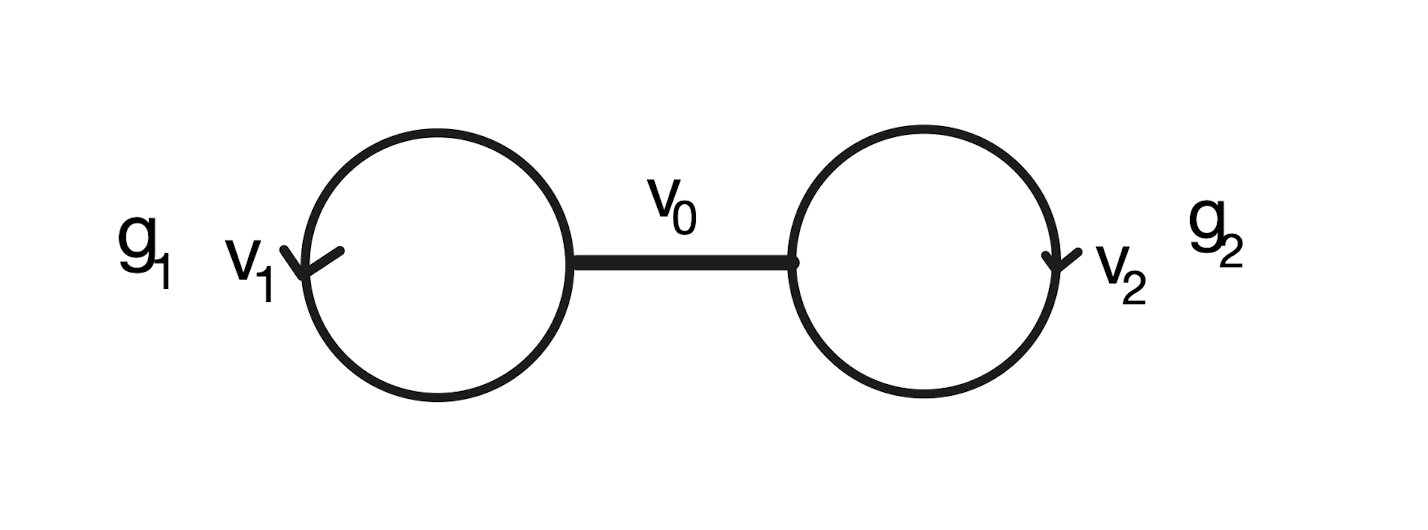}
\caption{$(\Gamma\setminus\cd)\cup e_0$}
\end{figure}
Traversing loops yields the relations at the two vertices
 \[ (g_1-I)v_1-v_0=0,\quad (g_2-I)v_2+v_0=0 \] 
hence
\[v_1=(g_1-I)^{-1}v_0,\quad v_2=-(g_2-I)^{-1}v_0.
\]
Hence invertiblity of $g_\gamma-I$ along any loop implies that $v_0$ uniquely determines $v_1$ and $v_2$, which in turn uniquely determines the vectors on every edge of $\gamma_1$ and $\gamma_2$.  All other edges are assigned the zero vector by the argument above.  

\begin{figure}[H]
\label{comptheta}
\centering
\includegraphics[scale=0.3]{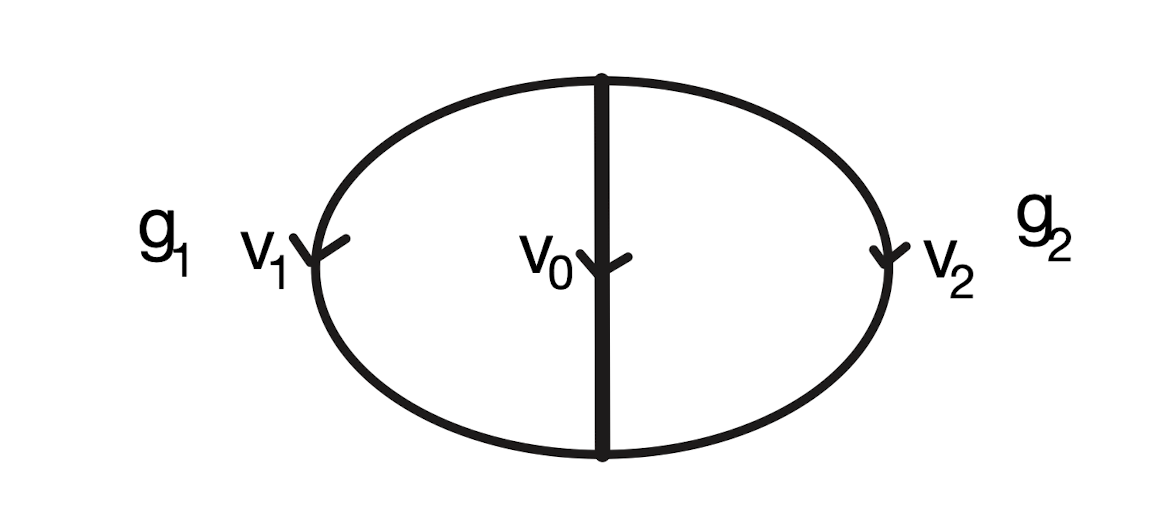}\includegraphics[scale=0.3]{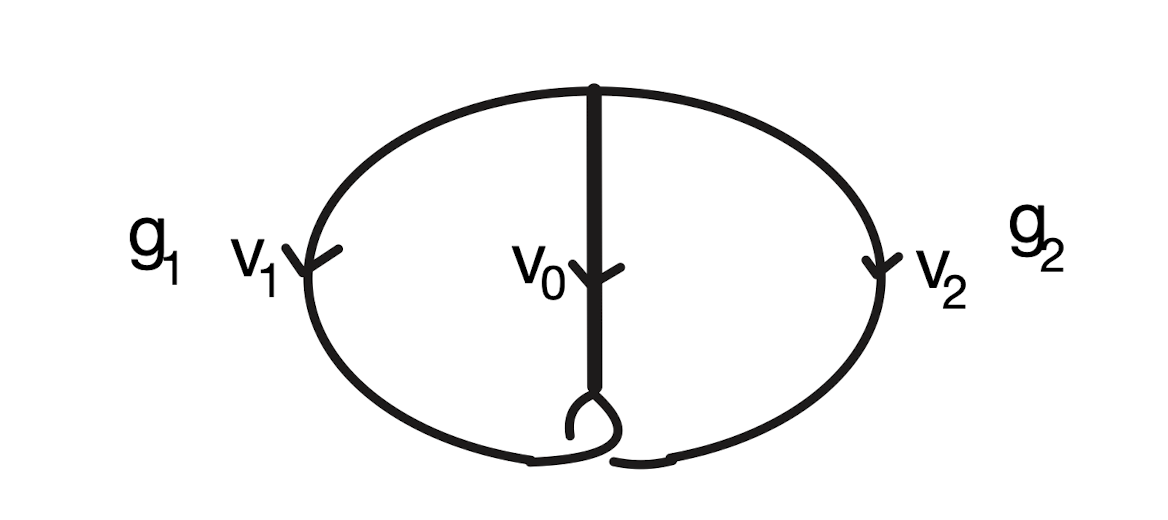}
\caption{$(\Gamma\setminus\cd)\cup e_0$}
\end{figure}
The two vertices gives the relations
\begin{equation} \label{thetarel}
v_0+v_1+v_2=0,\quad v_0+g_1v_1+g_2v_2=0
\end{equation} 
hence
\[v_1=-g_1^{-1}(g_2g_1^{-1}-I)^{-1}(g_2-I)v_0,\quad v_2=g_1^{-1}(g_2g_1^{-1}-I)^{-1}(g_1-I)v_0.
\]
Again $v_0$ uniquely determines $v_1$ and $v_2$, using invertibility of $g_2g_1^{-1}-I$ where $g_2g_1^{-1}=g_\gamma$ for a loop $\gamma$.

Thus each choice of $ v_{e_0}\in\mathbb R^2 $ determines a unique element of $\ker\delta$. The elements arising from distinct dimer edges are linearly independent, since each vanishes on every other dimer edge. Conversely, if an element of $\ker\delta$ vanishes on every dimer edge, then the above propagation argument shows that it vanishes identically. Consequently, each dimer edge contributes a two-dimensional subspace of $H_1(\Gamma,T_\Sigma^{1/2})$, and since $D$ consists of $2g-2+n$ edges, these subspaces together form a basis of $2(2g-2+n)$ vectors for $H_1(\Gamma,T_\Sigma^{1/2})$, proving \eqref{cohdim} under the assumption that $g_\gamma-I$ is invertible along all loops $\gamma$.\\

\noindent {\em 3. Holonomy with non-trivial fixed points.} 
Along a Ramond boundary component, the holonomy $g$ satisfies $\tr(g)>0$ and, since the hyperbolic metric is complete, it satisfies $\tr(g)=2$.  So 
\[g\sim\left(\begin{array}{cc}1&a\\0&1\end{array}\right),\qquad a\neq0\] 
and $g-I$ is not invertible with rank equal to one.  In this case, the argument above needs to be modified.

When $g>0$, one can still choose a ribbon graph $\Gamma$ with collection $\cd$ so that each boundary component contains at least two edges from $\cd$. This guarantees that $g_\gamma-I$ is invertible along the collection of loops $\gamma$ in $\Gamma\setminus\cd$ which implies \eqref{cohdim} by the argument above.   When $g=1$, a trivalent ribbon graph is given in the diagram.  The diagram shows a planar model of a genus one surface where arrows and labels indicate identification of sides.  The ribbon graph has $3n$ edges, $2n$ vertices and $n$ boundary components.  The dimer is given by the set of vertical edges, thickened in the diagram.
\begin{center}
\begin{tikzpicture}[scale=1.2,line cap=round,line join=round,>=stealth]
\centering
\def\r{1}

\usetikzlibrary{shapes.geometric}

\tikzset{
  hex/.style={
    draw,
    thick,
    regular polygon,
    regular polygon sides=6,
    minimum size={2*\r cm},
    rotate=30,
    inner sep=0pt
  },
  arr/.style={->,>=stealth,thick,shorten >=2pt,shorten <=2pt}
}

\node[hex] (H1) at (0,0) {};
\node[hex] (H2) at ({1.5*\r},0) {};
\node[hex] (H3) at ({3*\r},0) {};

\draw[line width=1mm] (-.75*\r,-.4*\r) -- (-.75*\r,.4*\r);
\draw[line width=1mm] (.75*\r,-.4*\r) -- (.75*\r,.4*\r);
\draw[line width=1mm] (2.25*\r,-.4*\r) -- (2.25*\r,.4*\r);
\draw[line width=1mm] (3.75*\r,-.4*\r) -- (3.75*\r,.4*\r);
\draw[line width=1mm] (5.6*\r,-.4*\r) -- (5.6*\r,.4*\r);
\draw[line width=1mm] (7*\r,-.4*\r) -- (7*\r,.4*\r);

\draw[line width=.8mm, ->]  (-.75*\r,-.1*\r) -- (-.75*\r,.1*\r);
\draw[line width=.8mm, ->]  (7*\r,-.1*\r) -- (7*\r,.1*\r);

\fill ({4.3*\r},0) circle (1pt);
\fill ({4.65*\r},0) circle (1pt);
\fill ({5.0*\r},0) circle (1pt);

\node[hex] (H4) at ({6.3*\r},0) {};

\foreach \x/\lone/\ltwo in {
  0/{a_1}/{a_2},
  {1.5*\r}/{a_3}/{a_4},
  {3*\r}/{a_5}/{a_6},
  {6.3*\r}/{a_{2n-1}}/{a_{2n}}
}{
  \begin{scope}[shift={(\x,0)}]
    \draw[arr] (-0.72*\r,0.42*\r)--(-0.32*\r,0.65*\r);
    \node at (-0.55*\r,0.95*\r) {$\lone$};

    \draw[arr] (0.32*\r,0.65*\r)--(0.72*\r,0.42*\r);
    \node at (0.55*\r,0.95*\r) {$\ltwo$};
  \end{scope}
}

\foreach \x/\lone/\ltwo in {
  0/{a_{2n}}/{a_1},
  {1.5*\r}/{a_2}/{a_3},
  {3*\r}/{a_4}/{a_5},
  {6.3*\r}/{a_{2n-2}}/{a_{2n-1}}
}{
  \begin{scope}[shift={(\x,0)}]
    \draw[arr] (-0.72*\r,-0.42*\r)--(-0.32*\r,-0.65*\r);
    \node at (-0.55*\r,-0.95*\r) {$\lone$};

    \draw[arr] (0.32*\r,-0.65*\r)--(0.72*\r,-0.42*\r);
    \node at (0.55*\r,-0.95*\r) {$\ltwo$};
  \end{scope}
}

\draw[arr] (-\r+.28, -0.1*\r)--(-\r+.29,0.1*\r); 
\node[left] at (-\r,0) {$a_0$};

\node[right] at ({6.1*\r+\r},0) {$a_0$};

\end{tikzpicture}
\end{center}
Each face contains two dimer edges, so if we set $v_e=0$ along all dimer edges $e$ except for $e_0$, then any loop $\gamma$ in $(\Gamma\setminus\cd)\cup e_0$ cannot be a boundary loop, hence $g_\gamma-I$ is invertible and the argument above applies.

For higher genus $g>1$, one can construct a ribbon graph $\Gamma$ with collection $\cd$ so that each boundary component contains at least two edges from $\cd$ as follows.  Begin with the genus 1 example above, then take two different faces and attach a new edge dimer to the centre of two non-dimer edges as in the diagram below. The new edge is pictured in two parts, where the two end points marked by $\bullet$ are identified and do not represent a vertex.  The new picture sends $(V,E,n)\mapsto (V+2,E+3,n-1)$ and $g\mapsto g+1$.  Each boundary component still contains at least two dimer edges, since two boundary components are combined into one and the others are unaffected. Apply this repeatedly to increase the genus, beginning with a large enough $n$.
\begin{center}
\begin{tikzpicture}[scale=1.2,line cap=round,line join=round,>=stealth]
\centering
\def\r{1}

\usetikzlibrary{shapes.geometric}

\tikzset{
  hex/.style={
    draw,
    thick,
    regular polygon,
    regular polygon sides=6,
    minimum size={2*\r cm},
    rotate=30,
    inner sep=0pt
  },
  arr/.style={->,>=stealth,thick,shorten >=2pt,shorten <=2pt}
}

\node[hex] (H1) at (0,0) {};
\node[hex] (H2) at ({1.5*\r},0) {};
\node[hex] (H3) at ({3*\r},0) {};

\draw[line width=1mm] (-.75*\r,-.4*\r) -- (-.75*\r,.4*\r);
\draw[line width=1mm] (.75*\r,-.4*\r) -- (.75*\r,.4*\r);
\draw[line width=1mm] (2.25*\r,-.4*\r) -- (2.25*\r,.4*\r);
\draw[line width=1mm] (3.75*\r,-.4*\r) -- (3.75*\r,.4*\r);
\draw[line width=1mm] (5.6*\r,-.4*\r) -- (5.6*\r,.4*\r);
\draw[line width=1mm] (7*\r,-.4*\r) -- (7*\r,.4*\r);

\draw[line width=.8mm, ->]  (-.75*\r,-.1*\r) -- (-.75*\r,.1*\r);
\draw[line width=.8mm, ->]  (7*\r,-.1*\r) -- (7*\r,.1*\r);

\draw[line width=.8mm]  (0,0) -- (-.3,.65);
\draw[line width=.8mm]  (1.5,0) -- (1.2,-.62);
\fill (0,0) circle (2pt);
\fill (1.5,0) circle (2pt);

\fill ({4.3*\r},0) circle (1pt);
\fill ({4.65*\r},0) circle (1pt);
\fill ({5.0*\r},0) circle (1pt);

\node[hex] (H4) at ({6.3*\r},0) {};

\foreach \x/\lone/\ltwo in {
  0/{a_1}/{a_2},
  {1.5*\r}/{a_3}/{a_4},
  {3*\r}/{a_5}/{a_6},
  {6.3*\r}/{a_{2n-1}}/{a_{2n}}
}{
  \begin{scope}[shift={(\x,0)}]
    \draw[arr] (-0.72*\r,0.42*\r)--(-0.32*\r,0.65*\r);
    \node at (-0.55*\r,0.95*\r) {$\lone$};

    \draw[arr] (0.32*\r,0.65*\r)--(0.72*\r,0.42*\r);
    \node at (0.55*\r,0.95*\r) {$\ltwo$};
  \end{scope}
}

\foreach \x/\lone/\ltwo in {
  0/{a_{2n}}/{a_1},
  {1.5*\r}/{a_2}/{a_3},
  {3*\r}/{a_4}/{a_5},
  {6.3*\r}/{a_{2n-2}}/{a_{2n-1}}
}{
  \begin{scope}[shift={(\x,0)}]
    \draw[arr] (-0.72*\r,-0.42*\r)--(-0.32*\r,-0.65*\r);
    \node at (-0.55*\r,-0.95*\r) {$\lone$};

    \draw[arr] (0.32*\r,-0.65*\r)--(0.72*\r,-0.42*\r);
    \node at (0.55*\r,-0.95*\r) {$\ltwo$};
  \end{scope}
}

\draw[arr] (-\r+.28, -0.1*\r)--(-\r+.29,0.1*\r); 
\node[left] at (-\r,0) {$a_0$};

\node[right] at ({6.1*\r+\r},0) {$a_0$};

\end{tikzpicture}
\end{center}

In genus 0 there does not exist such a pair $(\Gamma,\cd)$ so that each Ramond boundary component contains at least two edges from $\cd$.  This is simply due to the fact that for $g=0$, $2|D|=2(2g-2+n)<2n$.  For $n>3$, it can be arranged that each boundary component contains at least dimer edge.  Simply begin with the $n=4$ picture below.  Each boundary component contains one dimer edge and two non-dimer edges. 
\begin{center}
\begin{tikzpicture}[scale=.8,line cap=round,line join=round]
\tikzset{
  edge/.style={line width=0.8pt},
  dimer/.style={line width=2.6pt},
  vertex/.style={circle,fill=black,inner sep=1.8pt}
}

\coordinate (T) at (90:1.5);
\coordinate (L) at (210:1.5);
\coordinate (R) at (330:1.5);

\coordinate (V1) at (0,0);

\draw[dimer] (T) arc[start angle=90,end angle=210,radius=1.5];
\draw[edge]  (L) arc[start angle=210,end angle=330,radius=1.5];
\draw[edge]  (R) arc[start angle=330,end angle=450,radius=1.5];

\draw[edge]  (V1)--(T);
\draw[edge]  (V1)--(L);
\draw[dimer] (V1)--(R);

\end{tikzpicture}
\end{center}
For $n>4$, add a dimer edge inside a face from the centres of two non-dimer as in the diagram below.  This sends $(V,E,n)\mapsto (V+2,E+3,n+1)$  and each boundary component of the new planar ribbon graph contains one dimer edge and at least two non-dimer edges. So the process can continue to achieve any number $n$ of boundary components.  
\begin{center}
\begin{tikzpicture}[scale=.8,line cap=round,line join=round]
\tikzset{
  edge/.style={line width=0.8pt},
  dimer/.style={line width=2.6pt},
  vertex/.style={circle,fill=black,inner sep=1.8pt}
}

\coordinate (T) at (90:1.5);
\coordinate (L) at (210:1.5);
\coordinate (R) at (330:1.5);

\coordinate (U) at (90:.8);
\coordinate (V1) at (0,0);

\draw[dimer] (T) arc[start angle=90,end angle=210,radius=1.5];
\draw[edge]  (L) arc[start angle=210,end angle=330,radius=1.5];
\draw[edge]  (R) arc[start angle=330,end angle=450,radius=1.5];

\draw[dimer] (U) arc[start angle=90,end angle=210,radius=.8];
\draw[edge]  (V1)--(T);
\draw[edge]  (V1)--(L);
\draw[dimer] (V1)--(R);

\end{tikzpicture}
\end{center}
Now assume we have a genus 0 ribbon graph $\Gamma$ and a dimer $\cd$ such that each boundary component contains at least one dimer edge.  Then around any loop $\gamma\subset\Gamma\setminus\cd$ we have $g_\gamma-I$ is invertible since $\gamma$ is not a Ramond boundary.  Following the argument above, we take $e_0\in\cd$ and we an conclude that if $\gamma$ does not meet $e_0$ then $v_e=0$ along any edge of $\gamma$.  If $\gamma$ meets $e_0$, Figure~\ref{compdumb} and Figure~\ref{comptheta} show $(\Gamma\setminus\cd)\cup e_0$.  

In the case of Figure~\ref{compdumb}, neither of the two loops is a Ramond boundary, since any boundary loop contains a dimer edge, which would be absent from the figure.  Hence $g_1-I$ and $g_2-I$ are invertible and the earlier argument applies to conclude that any $v_0\in\br^2$ uniquely determines $v_1$ and $v_2$ thus a solution of $\delta v=0$.  

For the graphs in Figure~\ref{comptheta}, the loop $\gamma_{12}$ consisting of the edges $e_1$ and $e_2$ is not a boundary loop since it contains no dimer edges.  Hence $g_2g_1^{-1}-I$ is invertible and we can draw the same conclusions as the earlier argument, i.e. that $v_0\in\br^2$ uniquely determines $v_1$ and $v_2$.

Thus again we conclude that $D$ determines a vector space $(\br^2)^\cd$ of solutions of $\delta v=0$.  Furthermore, any solution is constructed this way since the dimer is chosen so that any solutions of $\delta v=0$ necessarily vanishes on all dimer edges.\\

\noindent {\em 4. Calculation of $H^1(\cc,L^\vee)$.}
The dimension
\begin{equation}  \label{dimh1}
h^1(\cc,L^\vee)=\dim H^1(\cc,L^\vee)=2g-2+\frac{1}{2}(n+|\sigma|)
\end{equation}
may be computed via the orbifold Riemann-Roch theorem which incorporates the isotropy representation at each orbifold point. It gives
\begin{align*}
h^0(\cc,\cl^\vee)-h^1(\cc,\cl^\vee)
&=
1-g+\deg \cl^\vee-\sum_{i=1}^n\lambda_{p_i}\\
&=
1-g+\left(1-g-\frac{n}{2}\right)-\frac{|\sigma|}{2} 
=
2-2g-\frac{1}{2}(n+|\sigma|),
\end{align*}
where $|\sigma|=\sigma_1+\cdots+\sigma_n$,
and $\lambda_{p_i}=\sigma_i/2$ is the contribution of the isotropy representation at $p_i$.  It is also useful to see this calculation on the coarse curve $C$ which gives 
\[h^0(C,L^\vee(-NS))-h^1(C,L^\vee(-NS))=2-2g-\tfrac12(n+|\sigma|)\] 
for $L=\rho_*\cl$ since $\rho_*(\cl^\vee)=(\rho_*\cl)^\vee(-NS)$.

We see that the dimensions of $H^1_{dR}(\Sigma,T_\Sigma^\frac12)$ and $H^1(\cc, \cl^\vee)^\vee$ disagree when $R\neq\emptyset$ and it remains to prove there relation via the exact sequence \eqref{ramex}.\\

\noindent{\em 5. Construction of Ramond boundary solutions.} For each Ramond boundary loop $\gamma\subset\Gamma$ with parabolic holonomy, there exists a non-trivial solution of $\delta v=0$, unique up to scale, supported on $\gamma$.  This defines a one-dimensional space of solutions to $\delta v=0$.  The solution is constructed as follows.  The holonomy $g_\gamma$ around $\gamma$ is parabolic with positive trace hence $g_\gamma-I$ is of rank one and $(g_\gamma-I)v=0$ defines a non-zero vector $v\in\br^2$ unique up to scale.     The edges of $\gamma$ are assigned $g_i\in SL(2,\br)$, $i=1,\ldots,m$, so that $g_m\cdots g_2g_1=g_\gamma$.

Suppose first that each edge of $\gamma$ meets a different boundary loop on the other side, or equivalently $\gamma$ meets each edge of $\Gamma$ at most once.  Assign $v_i\in\br^2$ to each edge of $\gamma$ where $v_1=v=gv$ as in the diagram.  Assign $0$ to all other edges.  
\begin{figure}[H]  \label{hexagon}
\begin{tikzpicture}[scale=.8,line cap=round,line join=round,>=stealth]
\usetikzlibrary{calc}

\tikzset{
  edge/.style={thick},
  dashededge/.style={thick,dashed},
  arr/.style={->,thick,>=stealth,shorten >=2pt,shorten <=2pt},
  vertex/.style={circle,fill=black,inner sep=2pt}
}

\coordinate (R)  at (0:1.6);
\coordinate (UR) at (60:1.6);
\coordinate (UL) at (120:1.6);
\coordinate (L)  at (180:1.6);
\coordinate (DL) at (240:1.6);
\coordinate (DR) at (300:1.6);

\draw[edge] (R)--(UR);
\draw[edge] (UR)--(UL);
\draw[edge] (UL)--(L);
\draw[dashededge] (L)--(DL);
\draw[edge] (DR)--(DL);
\draw[edge] (DR)--(R);

\draw[edge] (R)--++(0:0.8);
\draw[edge] (UR)--++(60:0.8);
\draw[edge] (UL)--++(120:0.8);
\draw[edge] (L)--++(180:0.8);
\draw[edge] (DL)--++(240:0.8);
\draw[edge] (DR)--++(300:0.8);

\draw[arr] ($(UL)!0.35!(UR)$)--($(UL)!0.65!(UR)$);
\draw[arr] ($(UR)!0.35!(R)$)--($(UR)!0.65!(R)$);
\draw[arr] ($(R)!0.35!(DR)$)--($(R)!0.65!(DR)$);
\draw[arr] ($(DR)!0.35!(DL)$)--($(DR)!0.65!(DL)$);
\draw[arr] ($(L)!0.35!(UL)$)--($(L)!0.65!(UL)$);

\foreach \P in {R,UR,UL,L,DL,DR}
  \fill (\P) circle (2pt);

\node[above] at ($(UL)!0.5!(UR)$) {$v_1$};
\node[left] at ($(L)!0.5!(UL)$) {$v_m$};
\node[right] at ($(UR)!0.5!(R)$) {$v_2$};
\node[right] at ($(R)!0.5!(DR)$) {$v_3$};

\node[right]      at ($(R)+(0:0.8)$)    {$0$};
\node[above right]at ($(UR)+(60:0.8)$)  {$0$};
\node[above left] at ($(UL)+(120:0.8)$) {$0$};
\node[left]       at ($(L)+(180:0.8)$)  {$0$};
\node[below left] at ($(DL)+(240:0.8)$) {$0$};
\node[below right]at ($(DR)+(300:0.8)$) {$0$};

\end{tikzpicture}
\caption{$v_1=v$, $v_2=g_1v_1$,\quad $v_3=g_2v_2$,\quad...,\quad $v_1=g_mv_m=g_\gamma v_1$}
\end{figure}
\noindent This gives a solution to $\delta v=0$ supported on $\gamma$ since at each vertex 
\[-g_kv_k+v_{k+1}+0=0\]
and $0+0+0=0$ at all vertices outside of $\gamma$, as required.

Now suppose instead the Ramond boundary loop $\gamma$ meets some edges of $\Gamma$ twice.  In this case, the solution supported on $\gamma$  is somewhat more involved.  Let $\{e_1,...,e_m\}\subset\gamma$ be those edges of $\Gamma$ that $\gamma$ meets exactly once, and refer to them as lone edges.  We refer to those edges that meet $\Gamma$ twice as paired edges.   The set of lone edges is necessarily non-empty, i.e. $m\geq1$, because evenness of $|R|$ implies that $\Gamma$ has at least two boundary components, and connectedness of $\Gamma$ implies that any boundary component must share an edge with at least one other boundary component.   

Figure~\ref{octagon} shows a paired edge, assigned $v_0$.  The holonomy around $\gamma$ is
\[g_\gamma=g_m...g_0^{-1}...g_2g_1g_0\] 
which is the word obtained by inserting for each paired edge $e$, holonomy $g_e$ and $g_e^{-1}$ inside the word $g_mg_{m-1}...g_2g_1$.     A paired edge $e$ is not a closed loop, since a boundary component must lie on one side of any closed loop, and more generally the union of the paired edges cannot contain a closed loop, again since a boundary component must lie on one side of any closed loop.  Hence, via a gauge transformation, we may arrange that $g_e=I$ for all paired edges $e$.  Thus $g_\gamma=g_m...g_2g_1$ gets contributions only from lone edges.  Let $v_1=v\neq 0$ ,where $g_\gamma v=v$  and define $v_{k+1}=g_kv_k$ for $k=1,...m-1$ so that $g_mv_m=v_1$.  Assign the vectors $v_1,...,v_m$ to the lone edges $\{e_1,...,e_m\}$ (oriented so that the boundary lies to the right).
\begin{figure}[H]   \label{octagon}
\begin{tikzpicture}[scale=1.3,line cap=round,line join=round,>=stealth]
\usetikzlibrary{calc}

\tikzset{
  edge/.style={thick},
  halfedge/.style={thick},
  arr/.style={->,thick,shorten >=2pt,shorten <=2pt},
  vertex/.style={circle,fill=black,inner sep=2pt}
}

\coordinate (A) at (-1,1);
\coordinate (B) at (1,1);
\coordinate (C) at (1.7,0.3);
\coordinate (D) at (1.7,-0.7);
\coordinate (E) at (1,-1.4);
\coordinate (F) at (-1,-1.4);
\coordinate (G) at (-1.7,-0.7);
\coordinate (H) at (-1.7,0.3);

\draw[edge]
(A)--node[midway,above]{$v_0$}
(B)--node[midway,right]{$v_1$}(C);
\draw[dashed] (C)--(D);
\draw[edge] (D)--node[pos=0.7,right]{$\ v_{k-1}$}(E)--node[midway,below]{$v_0$}(F)--node[pos=0.4,left]{$v_k$}(G);
\draw[dashed] (G)--(H);
\draw[edge] (H)--node[pos=0.7,left]{$v_m\ $}(A);

\draw[arr] ($(A)!0.38!(B)$)--($(A)!0.62!(B)$);
\draw[arr] ($(B)!0.38!(C)$)--($(B)!0.62!(C)$);
\draw[arr] ($(D)!0.38!(E)$)--($(D)!0.62!(E)$);
\draw[arr] ($(E)!0.62!(F)$)--($(E)!0.38!(F)$);
\draw[arr] ($(F)!0.38!(G)$)--($(F)!0.62!(G)$);
\draw[arr] ($(H)!0.38!(A)$)--($(H)!0.62!(A)$);

\draw[halfedge] (A)--++(135:0.4);
\draw[halfedge] (B)--++(45:0.4);
\draw[halfedge] (C)--++(0:0.4);
\draw[halfedge] (D)--++(-45:0.4);
\draw[halfedge] (E)--++(-45:0.4);
\draw[halfedge] (F)--++(-135:0.4);
\draw[halfedge] (G)--++(180:0.4);
\draw[halfedge] (H)--++(180:0.4);

\node[above left]  at ($(A)+(135:0.4)$) {$v_k$};
\node[above right] at ($(B)+(45:0.4)$) {$v_{k-1}$};
\node[right]       at ($(C)+(0:0.4)$) {$0$};
\node[right]       at ($(D)+(-45:0.4)$) {$0$};
\node[below right] at ($(E)+(-45:0.4)$) {$v_1$};
\node[below left]  at ($(F)+(-135:0.4)$) {$v_m$};
\node[left]        at ($(G)+(180:0.4)$) {$0$};
\node[left]  at ($(H)+(180:0.4)$) {$0$};

\foreach \P in {A,B,C,D,E,F,G,H}
  \fill (\P) circle (2pt);

\end{tikzpicture}
\caption{$v_0=g_mv_m-v_k$,\quad $v_1=v=g_\gamma v$,\quad $v_2=g_1v_1$,\quad...}
\end{figure}
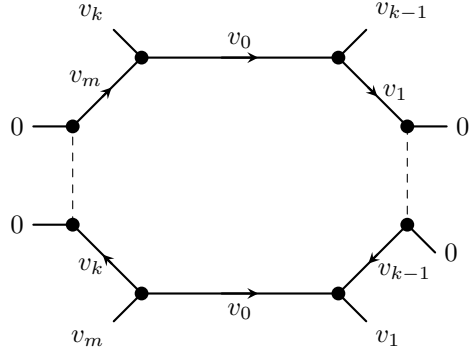
At a paired edge, assign a vector that satisfies the vertex condition at its source.  In Figure~\ref{octagon}, this means
\[ v_0=g_mv_m-v_k.
\]
Then the vertex condition at the target of the paired edge is automatically satisfied.  Again using  Figure~\ref{octagon} to demonstrate, and recalling that $g_0=I$, we get:
\[ v_0=g_mv_m-v_k=v_1-g_{k-1}v_{k-1}
\]
as required.  The same applies to all paired edges.\\

\noindent{\em 6. Linear independence of Ramond boundary solutions.} The construction above defines $|R|$ solutions supported at each Ramond boundary component which span a subspace.  In fact this gives an $|R|$-dimensional subspace of solutions because the solutions supported at each Ramond boundary component are linear independent which is proven as follows.  Suppose there is a non-trivial linear relation between solutions, each supported on a single Ramond boundary component.  If a Ramond boundary component contains a paired edge, then it doesn't appear in the linear relation since its coefficient is necessarily zero to remove its paired edge contribution.  Similarly, if a lone edge on a Ramond boundary component meets a Neveu--Schwarz boundary component, then it does not appear in the linear relation, since its coefficient is necessarily zero to remove its lone-edge contribution.   The remaining case is an edge $e$ that meets two different Ramond boundary components $\gamma_1$ and $\gamma_2$.  If there is a linear relation, then the $e$ component can be canceled only if the two solutions supported on each of $\gamma_1$ and $\gamma_2$ define parallel vectors on $e$.  But then there exists $v\neq0$ such that $\gamma_1v=v=\gamma_2v$.  The fixed vector corresponds to a fixed point in $\partial\mathbb H$.  Two parabolic elements have the same fixed
point in $\partial\mathbb H$ if and only if they lie in the same
maximal parabolic subgroup. Therefore parabolic elements associated to
distinct cusps have distinct fixed points, contradicting the existence of $v$.  This proves the linear independence of the solutions supported at each Ramond boundary component hence they define an $|R|$-dimensional set of solutions.

Thus $\dim H^1_{dR}(\Sigma,T_\Sigma^\frac12)-|R|=\dim H^1(\cc, \cl^\vee)^\vee$ which gives an exact sequence
\[0\to\br^R\to H^1_{dR}(\Sigma,T_\Sigma^\frac12)\to H^1(\cc, \cl^\vee)^\vee\to0\]
as required.
\end{proof}

\subsection{Vector bundle $E_{g,n}\to{\cal M}_{g,n}^{\text{spin}}$.}   \label{vecE}

The sheaf cohomology groups calculated in \ref{sheafcoh} associate a vector space to each spin curve  which together define a vector bundle over the moduli space of spin curves.  The vector bundle arises as a derived pushforward of a sheaf.

Denote the universal spin curve by $\cu_{g,\sigma}^{\text{spin}}\stackrel{\pi}{\longrightarrow}{\cal M}_{g,\sigma}^{\text{spin}}$ for $\sigma=(\sigma_1,...,\sigma_n)$.  Let $\ce$ be the universal spin bundle defined over $\cu_{g,\sigma}^{\text{spin}}$ so that $\ce^{\otimes 2}\cong\omega_\pi^{\log}$.  

 \begin{definition}  \label{obsbun}
Define a vector bundle by $E_{g,n}^\vee:=-R\pi_*\ce^\vee\hspace{-1mm}\to{\cal M}_{g,\sigma}^{\text{spin}}$. 
\end{definition} 
This defines a holomorphic vector bundle with fibre 
\[
E_{g,n}|_{[\cc]}=H^1(\cc,\cl^{\vee})^\vee
\]
for $\cl^{\otimes 2}\cong \omega_\cc$.  It is a vector bundle since $\deg \cl^{\vee}=1-g-\frac12n<0$ so $H^0(\cc,\cl^{\vee})=0$ and $H^1(\cc,\cl^{\vee})$ has constant dimension over connected components of ${\cal M}_{g,n}^{\text{spin}}$.   The dual $\ce^\vee$ in the definition is important, since without it we would instead produce a virual bundle built from both $H^0(\cc,L)$ and $H^1(\cc,L)$ since both are in general non-trivial.  The dual present in $E_{g,n}^\vee$ is unimportant, since it only introduces a factor of $(-1)^n$ into later formulae, but it turns out to be more natural.

Denote the restriction to each component by $E_{g,\sigma}=E_{g,n}|_{{\cal M}_{g,\sigma}^{\text{spin}}}$.  The  rank of $E_{g,\sigma}$ is obtained from the Riemann-Roch calculation in \ref{sheafcoh}:
\[
\text{rank\ }E_{g,\sigma}=2g-2+n-\tfrac12|R|
\]
where $\sigma=(1^{n-|R|},0^{|R|})$.

Now instead we give a construction of a vector bundle using the hyperbolic viewpoint which uses the derived pushforward of the universal spin structure given by a relative local system.  Define
\[
\mathcal{U}= \Sigma\times \ct_{g,n}^{\rm spin}.
\]
Each point in Teichm\"uller space $b\in\ct_{g,n}^{\rm spin}$ determines a Fuchsian representation up to conjugation
\[
\varrho_b:\pi_1(\Sigma)\longrightarrow SL(2,\br),
\]
and we choose conjugacy class representatives so that $\varrho_b$ varies continuously in $b$.  

The family of representations $\varrho_b$ defines a universal spin structure which is a rank two real vector bundle
\[
\be=\big(\bh\times\br^2\times\ct_{g,n}^{\rm spin}\big)\big/\pi_1(\Sigma)
\longrightarrow
\mathcal{U},
\]
where the action of $\gamma\in\pi_1(\Sigma)$ is given by
\[
\gamma\cdot(x,v,b)
=
\big(\gamma x,\varrho_b(\gamma)v,b\big).
\]
The restriction of $\be$ to the fibre $\mathcal{U}_b=\Sigma\times\{b\}$
is precisely the flat bundle associated to the representation
$\varrho_b:\pi_1(\Sigma)\longrightarrow SL(2,\br)$,
\[ \be|_{\mathcal{U}_b}\cong T_{\mathcal{U}_b}^{\frac12}
\] 
Define the relative local system $\ce^F$ to be the sheaf of sections of $\be$ locally constant along fibres $\mathcal{U}_b$.  
 \begin{definition}  \label{obsbunflat}
Define the vector bundle $E^F_{g,n}:=-R\pi_*\ce^F\hspace{-1mm}\to{\cal M}_{g,n}^{\text{spin}}$. 
\end{definition} 
This defines a vector bundle with fibre 
\[
E^F_{g,n}|_{[\Sigma]}= H^1_{dR}(\Sigma,T_\Sigma^\frac12).
\]
The restriction to each component is denoted by $E^F_{g,\sigma}=E^F_{g,n}|_{{\cal M}_{g,\sigma}^{\text{spin}}}$.  The calculation in the proof of Theorem~\ref{exactseq} gives
\[
\text{rank\ }E^F_{g,n}=4g-4+2n.
\]
The rank of $E^F_{g,n}$ is the same on each component of ${\cal M}_{g,n}^{\text{spin}}$ unlike the rank of $E_{g,n}$.

Over the Neveu--Schwarz components ${\cal M}_{g,\sigma}^{\text{spin}}$ for $\sigma=(1^n)$, i.e. all labeled points of the spin structure are of Neveu--Schwarz type, the isomorphism \eqref{canis} is canonical and hence the bundles are isomorphic
\[ E_{g,\sigma}\cong E^F_{g,\sigma},\qquad \sigma=(1^n).
\]
More generally, we expect an isomorphism of $E_{g,\sigma}$ with $E^F_{g,\sigma}/\co_\cu^{\oplus R}$, but the construction in Theorem~\ref{exactseq} has not yet been proven to be canonical.
\begin{exercise} {\bf Open problem.}
Prove
\[E_{g,\sigma}\cong E^F_{g,\sigma}/\co_\cu^{\oplus R}?\]
\end{exercise}

\section{Spin measures on ${\cal M}_{g,n}$.}   \label{spinmeas}
\subsection{Canonical Euler form}   \label{caneul}
The bundle $E_{g,n}\to{\cal M}_{g,n}^{\text{spin}}$ possesses a canonical Euler form $e(E_{g,n})$.  Its construction uses a canonical Hermitian metric on $E_{g,n}$, defined similarly to the construction of the Weil--Petersson metric.  We first define the Weil--Petersson metric and its K\"ahler form, the Weil--Petersson form.  This will help to make the similarities clear.

For $[C]\in{\cal M}_{g,n}$, the deformation theory of a curve and Serre duality give
\[ T_{[C]}{\cal M}_{g,n}\cong H^1(C,T_C(-D))\cong H^0(C,\omega_C^{\otimes 2}(D))^\vee.
\]
Instead, one can avoid the need for poles by using twisted curves $\cc$,
\[ T_{[\cc]}{\cal M}_{g,n}^{(2)}\cong H^1(\cc,T_\cc)\cong H^0(\cc,\omega_\cc^{\otimes 2})^\vee.
\]
\begin{exercise}
Prove that $H^0(\cc,\omega_\cc^{\otimes 2})\cong H^0(C,\omega_{ribbon graph}^{\otimes 2}(D))$.
\end{exercise}
The pairing of cotangent vectors is obtained by integration over $\Sigma=C-D$.  For
\[\eta_1,\eta_2\in H^0(C,\omega_{C}^{\otimes 2}(D))\cong H^1(C,T_C(-D))^\vee\]
define
\begin{equation}  \label{WPmetric} 
\langle\eta_1,\eta_2\rangle:=\int_\Sigma\frac{\overline{\eta_1}\eta_2}{h}
\end{equation}
where $h$ is the complete hyperbolic metric on $\Sigma$.
This defines a Hermitian metric on $T_{[C]}{\cal M}_{g,n}$ (or its dual).  Its real part is a K\"ahler metric, the Weil--Petersson metric, and its imaginary part is the 
Weil--Petersson symplectic form $\omega_{WP}$.   Its associated measure $\exp\omega_{WP}$ is the Weil--Petersson measure.

We can give a similar construction in the presence of a spin structure to produce a Hermitian metric on $E_{g,n}$.  Let $(\cc,D,\cl,\phi)$ be a smooth, genus $g$, pointed spin curve.  Its coarse curve $C=|\cc|$ is equipped with $L=\rho_*\cl\to C$ which satisfies $L^{\otimes 2}\cong \omega_C^{\text{log}}(-NS)=\omega_C(R)$, for $D=NS\sqcup R$.  Serre duality applied to the fibres of $E_{g,n}$ gives
\[H^1(C,L^\vee(-NS))^\vee \cong H^0(C,\omega_C^{3/2}(NS)).\] 
The spin structure, which gives a choice of $z'(w)^{1/2}$ on overlaps of a cover of $C$, ensures that $3/2$ differentials are well-defined.  The $3/2$ differentials give the analogue of holomorphic quadratic differentials used above to define the Weil--Petersson metric.  For
\[ \eta,\xi\in H^0(C,\omega_C^{3/2}(NS))
\]
define a Hermitian metric via integration over $\Sigma=C-D$ equipped with a complete hyperbolic metric $h$,
\begin{equation}  \label{hermetric} 
\langle\eta,\xi\rangle:=\int_\Sigma\frac{\overline{\eta}\xi}{\sqrt{h}}
\end{equation}
where $\sqrt{h}$ is the hyperbolic metric on the spin bundle $T^{\frac12}_\Sigma$.   The hyperbolic metric $h$ on $\Sigma$ induces the metric $\sqrt{h}$.  In local coordinates simply take square roots $h=H(z)|dz|^2\mapsto\sqrt{H(z)}|dz|=\sqrt{h}$ where $H(z)>0$ and $\sqrt{H(z)}>0$.    Alternatively, the metric $\sqrt{h}$ can be obtained directly, without requiring a square root, via Hitchin's proof of the existence of a hyperbolic metric in a given conformal class \cite{HitSel}.  Hitchin's construction goes the other way, by first producing a Hermitian metric on the spin bundle which then induces a hyperbolic metric on the tangent bundle.

\subsubsection{Convergence}
The Hermitian metric \eqref{hermetric} is well-defined because the integral converges, which uses the following local argument.
If $\Sigma$ is compact the integral \eqref{hermetric} clearly converges.  When  $\Sigma$ is non-compact, i.e. $D\neq\varnothing$, to see that the integral exists, consider a local coordinate $z$ with $z=0$ corresponding to a point of $D$ and a cusp of the metric.  Locally, the hyperbolic metric is given by 
$h=\frac{|dz|^2}{|z|^2(\log|z|)^2}$
and the $3/2$ differentials are given by $\eta=\frac{f(z)dz^{3/2}}{z}$ and $\xi=\frac{g(z)dz^{3/2}}{z}$ where $f(z)$ and $g(z)$ are holomorphic at $z=0$.  The local contribution to the metric $\int_{|z|<\epsilon}\frac{\overline{f}g\log|z||dz|^2}{|z|}$ exists since 
\begin{equation}  \label{locest}
\int_{|z|<\epsilon}\frac{|\log|z||}{|z|}|dz|^2=\int_0^\epsilon|\log r| drd\theta=2\pi|\epsilon\log\epsilon-\epsilon|<2\pi \quad\Leftarrow\quad\epsilon<1.
\end{equation}

\subsubsection{Characteristic differential form} The bundle $E_{g,n}$ is holomorphic and together with the Hermitian metric on $E_{g,n}$, it uniquely determines the Chern connection, a metric connection $A$ on $E_{g,n}$ satisfying $\overline{\partial}_A=\overline{\partial}$ the natural operator defining the holomorphic structure on $E_{g,n}$.  Then $e(E_{g,n})$ is defined to be the Pfaffian of the curvature of $A$ via 
\begin{equation}  \label{eulerform}
e(E_{g,n}):=\left(\frac{1}{4\pi}\right)^N\text{pf}(F_A).
\end{equation}
The Weil--Petersson form $\omega_{WP}$ pulls back to ${\cal M}_{g,n}^{\text{spin}}$ under $p:{\cal M}_{g,n}^{\text{spin}}\to{\cal M}_{g,n}$, the map that forgets the spin structure.  Alternatively, it is defined directly via Goldman's symplectic form \cite{GolSym} on the character variety.  By abuse of terminology it is still denoted by $\omega_{WP}$.
\begin{thm}[\cite{NorSup}]  \label{volequal}
The differential form
\[\mu=e(E_{g,n})\exp\omega_{WP}\] 
defines a finite measure on ${\cal M}_{g,n}^{\text{spin}}$.
\end{thm}
Push forward the restriction of this measure to each component via  \[p^\sigma:{\cal M}_{g,\sigma}^{\text{spin}}\to{\cal M}_{g,n}\]  for $p^\sigma=p|_{{\cal M}_{g,\sigma}^{\text{spin}}}$ to get a collection of measures over ${\cal M}_{g,n}$.  We further normalise the measure by $\epsilon_{g,\sigma}=(-1)^n2^{g-1+\frac12(n+|\sigma|)}$.
\begin{definition}  \label{defmeas} 
For each $\sigma\in\{0,1\}^n$, define a spin measure on ${\cal M}_{g,n}$ by
\[\mu_\sigma:=\epsilon_{g,\sigma}p^\sigma_*(\mu)\in\Omega^{\text{top}}{\cal M}_{g,n}.
\]
\end{definition}
Note that since $p$ is a finite covering map, a top degree differential form pushes forward to a top degree differential form, which agrees with the natural pushforward measure.

The Weil--Petersson form can be expressed explicitly in terms of natural coordinates on Teichm\"uller space, such as Penner coordinates \cite{PenDec} or Fenchel-Nielsen coordinates \cite{WolSym}, as in \eqref{WPFN} below. 
\begin{exercise}  {\bf Open problem.}
 Express the measure $\mu_\sigma$ in terms of Penner coordinates or Fenchel-Nielsen coordinates.
\end{exercise}
The total measure defines a volume which coincides with the super Weil--Petersson volume of the moduli space of super Riemann surfaces \cite{NorSup,SWiJTG}.  The restriction to components gives the following collection of measures.
\begin{definition}
\[\widehat{V}_{g,\sigma}:=\int_{{\cal M}_{g,n}} \mu_\sigma,\quad\sigma\in\{0,1\}^n.
\]
\end{definition}
Theorem~\ref{disk} produces calculations of some of these volumes, such as
\[ \widehat{V}_{0,\{1,0^4\}}=6\pi^2,\qquad \widehat{V}_{0,\{1,0^6\}}=330\pi^4.
\]

\subsubsection{Hyperbolic construction of characteristic forms}
The construction of a canonical Euler form in \eqref{eulerform} uses the conformal structure on a hyperbolic surface $\Sigma$ via the appearance of $3/2$ differentials.  A spin hyperbolic surface may be constructed geometrically by gluing hyperbolic pairs of pants along geodesic boundary components, or analytically from a discrete faithful representation of its fundamental group into $SL(2,\br)$. In either description, the induced conformal structure is obtained only through the uniformization theorem and is therefore transcendental in nature. 

Here we instead describe a construction of an Euler form entirely in terms of hyperbolic geometry, using twisted harmonic $1$-forms.  Conjecturally this construction produces the same measure on ${\cal M}_{g,\sigma}^{\text{spin}}$.

For simplicity, we will consider the case when $\Sigma$ is compact.  Given a spin structure on $\Sigma$, which gives rise to the associated flat vector bundle $T_\Sigma^\frac12$, denote the flat connection on $T_\Sigma^\frac12$ by $\nabla_h$.  
Since $\nabla_h^2=0$, the flat connection defines the twisted de Rham complex
\[
0\longrightarrow \Omega^0(\Sigma,T_\Sigma^\frac12)
\xrightarrow{\nabla_h}
\Omega^1(\Sigma,T_\Sigma^\frac12)
\xrightarrow{\nabla_h}
\Omega^2(\Sigma,T_\Sigma^\frac12)
\longrightarrow0
\]
The twisted Poincar\'e lemma identifies the sheaf cohomology of the locally constant sheaf defined by $T_\Sigma^\frac12$ with the cohomology of this complex.  In particular:
\[
H^1_{dR}(\Sigma,T_\Sigma^\frac12)
\cong
\frac{\ker\bigl(\nabla_h:\Omega^1(\Sigma,T_\Sigma^\frac12)\to\Omega^2(\Sigma,T_\Sigma^\frac12)\bigr)}
{\operatorname{im}\bigl(\nabla_h:\Omega^0(\Sigma,T_\Sigma^\frac12)\to\Omega^1(\Sigma,T_\Sigma^\frac12)\bigr)}.
\]
The hyperbolic metric on $\Sigma$ induces a metric on $T_\Sigma^\frac12$. Let $\nabla_h^*$ denote the formal adjoint of $\nabla_h$, and define the twisted Laplacian
\[
\Delta_h=\nabla_h\nabla_h^*+\nabla_h^*\nabla_h.
\]
By Hodge theory, every class in $H^1_{dR}(\Sigma,T_\Sigma^\frac12)$ has a unique harmonic representative. Thus
\[
H^1_{dR}(\Sigma,T_\Sigma^\frac12)
\cong
\ch^1(\Sigma,T_\Sigma^\frac12)
\]
where
\[
\ch^1(\Sigma,T_\Sigma^\frac12)
=
\left\{
\eta\in\Omega^1(\Sigma,T_\Sigma^\frac12)
\;\middle|\;
\nabla_h\eta=0,\  \nabla_h^*\eta=0
\right\}
=
\ker\Delta_h\big|_{\Omega^1(\Sigma,T_\Sigma^\frac12)}.
\]
The bundle $E_g=E_{g,0}$ has fibres $\ch^1(\Sigma,T_\Sigma^\frac12)$ and naturally sits inside the trivial bundle with fibres $\Omega^1(\Sigma,T_\Sigma^\frac12)$. There is a natural connection $\alpha$ on this bundle obtained by orthogonally projecting the derivative of a local harmonic section back onto the harmonic subspace of $\Omega^1(\Sigma,T_\Sigma^\frac12)$.  

Modify the construction of the Euler form \eqref{eulerform} by replacing the connection $A$ with the connection $\alpha$ to produce a new form.
\begin{definition}
Define a characteristic form on $E_g$ by
\begin{equation}  \label{charform}
\xi(E_g):=\left(\frac{1}{4\pi}\right)^N\text{pf}(F_\alpha).
\end{equation}
\end{definition}
\begin{exercise} {\bf Open problem.}\\
Prove that there exists a differential form $\eta\in\Omega^{4g-5}({\cal M}_g^{\text{spin}},\br)$
\[ e(E_g)-\xi(E_g)=d\eta
\]
is exact  and hence they defines the same total measure
\[\int_{{\cal M}_g^{\text{spin}}}e(E_{g})\exp\omega_{WP}=\int_{{\cal M}_g^{\text{spin}}}\xi(E_{g})\exp\omega_{WP}.\] 
\end{exercise}

 \subsubsection{Hyperbolic surfaces with boundary}
The preceding construction extends to a compact hyperbolic surface $\Sigma$ with
geodesic boundary
\[
\partial\Sigma=\beta_1\sqcup\cdots\sqcup\beta_n.
\]
Fix a spin structure and let $T_\Sigma^{1/2}$ be the associated flat rank two
bundle with flat connection $\nabla_h$.  The hyperbolic metric and the metric on
$T_\Sigma^{1/2}$ again define the formal adjoint $\nabla_h^*$ and twisted
Laplacian
\[
\Delta_h=\nabla_h\nabla_h^*+\nabla_h^*\nabla_h.
\]
In the presence of boundary, harmonic representatives are obtained by imposing
Neumann boundary conditions \cite{SchHod}.  Define
\[
\ch^1(\Sigma,T_\Sigma^{1/2})
=
\left\{
\eta\in\Omega^1(\Sigma,T_\Sigma^{1/2})
\;\middle|\;
\nabla_h\eta=0,\ 
\nabla_h^*\eta=0,\ 
\iota_\nu\eta|_{\partial\Sigma}=0
\right\}.
\]
By Hodge theory for manifolds with boundary,
\[
H^1_{dR}(\Sigma,T_\Sigma^{1/2})
\cong
\ch^1(\Sigma,T_\Sigma^{1/2}).
\]
\begin{exercise}
Let $\Sigma$ be a compact Riemannian surface with boundary and let $E$ be a
flat orthogonal bundle.  Show, using integration by parts, that if
\[
\nabla\eta=0,\qquad \nabla^*\eta=0,
\qquad
\iota_\nu\eta|_{\partial\Sigma}=0,
\]
then $\eta$ is a harmonic $1$-form satisfying Neumann boundary conditions.
Explain why these harmonic forms represent $H^1(\Sigma,E)$ rather than
$H^1(\Sigma,\partial\Sigma;E)$.
\end{exercise}
Now fix the geodesic boundary lengths $L_1,\ldots,L_n$.  As $\Sigma$ varies in
the corresponding moduli space of spin hyperbolic surfaces, the spaces
$\ch^1_{\mathrm{abs}}(\Sigma,T_\Sigma^{1/2})$ form a vector bundle
\[
E_{g,n}(L_1,\ldots,L_n)
\longrightarrow
{\cal M}_{g,n}^{\mathrm{spin}}(L_1,\ldots,L_n).
\]
As in the closed case, orthogonal projection of the derivative of a local
family of twisted harmonic forms onto the harmonic subspace defines a
connection $\alpha_L$ on $E_{g,n}(L_1,\ldots,L_n)$.  We therefore obtain the
characteristic form
\[
\xi\bigl(E_{g,n}(L_1,\ldots,L_n)\bigr)
:=
\left(\frac{1}{4\pi}\right)^N
\operatorname{pf}(F_{\alpha_L}).
\]
This gives a construction depending only on the hyperbolic metric, its
geodesic boundary, and the flat bundle determined by the spin structure, and
hence avoids passing through the conformal description by $3/2$-differentials.  It would be useful to calculate it for a hyperbolic pair of pants which is the building block of general hyperbolic surfaces. 
It is natural to conjecture that this characteristic form represents the same
Euler class, and produces the same measure, as the Euler form obtained from
the holomorphic construction.

\subsection{Deformation of the Weil--Petersson measure}  \label{defWP}

There is a natural family of deformations of the Weil--Petersson measure obtained via deforming cusps to geodesic boundary components.  This was used by Mirzakhani \cite{MirSim} to define and calculate Weil--Petersson volumes depending on the lengths of the geodesic boundary components.  For $\mathbf L=(L_1,\ldots,L_n)\in\mathbb R_{\geq 0}^n$, let
\begin{align*}
{\cal M}_{g,n}(\mathbf L)=\Big\{(\Sigma,\beta_1,&...,\beta_n)\mid \Sigma \text{ oriented, genus }g\text{ hyperbolic surface,}\\ 
&\text{geodesic boundary } \partial \Sigma=\sqcup\beta_i\text{ of lengths }\ell(\beta_i)=L_i\Big\}/\sim
\end{align*}
The equivalence is up to isometries that preserve the boundary components.  A non-trivial isometry rotates each boundary component non-trivially.  When $\mathbf L=0$, the boundary components are replaced by cusps and
\[
{\cal M}_{g,n}(0)={\cal M}_{g,n}.
\]
\subsubsection{Fenchel-Nielsen coordinates}
Let $\Sigma$ be an oriented surface of genus $g$ with $n$ labeled boundary components.  Choose a pants decomposition
\[
\Sigma=P_1\cup\cdots\cup P_{2g-2+n}
\]
where $P_i\cong S^2-\text{ three disjoint disks}$ is a pair of pants.
The decomposition is obtained by cutting along $3g-3+n$
disjoint simple closed curves $\gamma_1,\ldots,\gamma_{3g-3+n}$. 
Given a hyperbolic structure on $\Sigma$ with geodesic boundary, each $\gamma_i$ has a unique geodesic representative of length
$
\ell_i>0.
$
The hyperbolic metric on each pair of pants is uniquely determined by the lengths of its three boundary components. To reconstruct the surface, the pairs of pants are glued along boundary components of equal length. Each such gluing has one additional real parameter
$
\tau_i\in\mathbb R,
$
called the \emph{twist parameter}, which measures the relative displacement of the seams of the pants on either side.

For fixed boundary lengths $\mathbf L=(L_1,\ldots,L_n)$, the resulting length and twist parameters
\[
(\ell_1,\tau_1,\ldots,\ell_{3g-3+n},\tau_{3g-3+n})
\]
give global coordinates on the Teichm\"uller space $\mathcal T_{g,n}(\mathbf L)$. These are the \emph{Fenchel--Nielsen coordinates}, giving a diffeomorphism
\[
\mathcal T_{g,n}(\mathbf L)
\cong
\mathbb R_{>0}^{\,3g-3+n}\times\mathbb R^{\,3g-3+n}.
\]
The symplectic form on $\mathcal T_{g,n}(\mathbf L)$ is defined by
\begin{equation} \label{WPFN}
\omega_{\mathbf L}:=\sum_id\ell_i\wedge d\tau_i.
\end{equation} 
When $\mathbf L=0$, Wolpert \cite{WolSym} proved that \eqref{WPFN} defines the Weil--Petersson form $\omega_{WP}$, and in particular, it is invariant under the action of the mapping class group so descends to ${\cal M}_{g,n}$.  By doubling along geodesic boundary components, it follows from Wolpert's proof that \eqref{WPFN} is invariant under the action of the mapping class group so $\omega_{\mathbf L}$ descends to ${\cal M}_{g,n}(\mathbf L)$.

There is a diffeomorphism
\[
\phi_{\mathbf L}:{\cal M}_{g,n}\to{\cal M}_{g,n}(\mathbf L)
\]
with inverse obtained by grafting a semi-infinite flat cylinder onto each geodesic boundary component to produce a punctured Riemann surface \cite{MonRie}.  Using this, $\omega_{\mathbf L}$ may be regarded as a family of symplectic forms on the single space ${\cal M}_{g,n}$.  
\begin{definition}
Define
\[
\omega(L_1,\ldots,L_n)
:=
\phi_{\mathbf L}^*\omega_{\mathbf L}
\in\Omega^2({\cal M}_{g,n}).
\]
\end{definition}
By construction,
\[
\omega(0,\ldots,0)=\omega_{WP}.
\]

\subsubsection{Deformation of spin measures}

It is straightforward now to insert the deformation of the Weil--Petersson symplectic form into Definition~\ref{defmeas} to produce a family of measures depending on $(L_1,...,L_n)$.  Let
\[ \mu(L_1,...,L_n):=e(E_{g,n})\exp\omega_{WP}(L_1,...,L_n)\] 
and normalise by $\epsilon_{g,n,m}=(-1)^{n+m}2^{g-1+n+\frac12m}$.
\begin{definition}  \label{defmeasL}
Define a family of spin measures on ${\cal M}_{g,n+m}$ by
\[\mu^{(m)}(L_1,...,L_n):=\epsilon_{g,n,m}p^\sigma_*\big(\mu(L_1,...,L_n)\big)\in\Omega^{\text{top}}{\cal M}_{g,n+m}
\]
where $\sigma=(1^n,0^m)$.  Each total measure defines the volume
\begin{equation}   \label{voldef}
\widehat{V}_{g,n}^{(m)}(L_1,...,L_n):=\int_{{\cal M}_{g,n+m}}\hspace{-5mm} \mu^{(m)}(L_1,...,L_n).
\end{equation}
\end{definition}
It is proven in \cite{NorSup} that the measures $\mu^{(m)}(L_1,...,L_n)$ are finite and the volumes $\widehat{V}_{g,n}^{(m)}(L_1,...,L_n)$ are polynomial in $L_1,...,L_n$.  Calculations of $\widehat{V}_{g,n}^{(m)}(L_1,...,L_n)$, and many properties, currently use intersection theory in algebraic geometry, described in Section~\ref{compact}.  A second approach due to Stanford and Witten \cite{SWiJTG}, and described in Section~\ref{super}, uses supergeometry techniques which follow the method of Mirzakhani's calculation of Weil--Petersson volumes \cite{MirSim}.  Johnson \cite{JohRam} computes the super volumes with Ramond punctures in a different way using matrix model techniques.   One aim of this paper is to lay the foundations of a proof using differential geometric methods.

\subsection{Relation to supergeometry}   \label{super}

We briefly describe the supergeometric derivation of the super volumes by Stanford--Witten \cite{SWiJTG}.

\subsubsection{Supermanifolds}

A locally ringed space $(M,\cf)$ consists of a topological space $M$ together with a sheaf of rings $\cf$ whose stalks are local rings.  A fundamental example is the smooth manifold $\br^m$ equipped with its sheaf $C^\infty_{\br^m}$ of locally smooth functions.  Its supercommutative analogue is
\[
\br^{m|n}
=
(\br^m,\co_{\br^{m|n}}),
\qquad
\co_{\br^{m|n}}
=
C^\infty_{\br^m}\otimes\Lambda^*(\br^n).
\]
A real supermanifold of dimension $m|n$ is a locally ringed space
\[
\widehat M=(M,\co_{\widehat M})
\]
which is locally isomorphic to $\br^{m|n}$.  The topological space $M$ is called the \emph{reduced space} or \emph{body} of $\widehat M$.

Similarly, define
\[
\bc^{m|n}
=
(\bc^m,\co_{\bc^{m|n}}),
\qquad
\co_{\bc^{m|n}}
=
\co_{\bc^m}\otimes\Lambda^*(\bc^n),
\]
where $\co_{\bc^m}$ is the sheaf of locally holomorphic functions.  A complex supermanifold is a locally ringed space locally isomorphic to $\bc^{m|n}$.

A morphism
\[
(f,F):
(M_1,\co_{\widehat M_1})
\longrightarrow
(M_2,\co_{\widehat M_2})
\]
consists of a continuous map $f:M_1\to M_2$ together with a graded sheaf homomorphism
\[
F:\co_{\widehat M_2}\longrightarrow f_*\co_{\widehat M_1}.
\]
A family of supermanifolds is a morphism of supermanifolds
\[
\widehat M\longrightarrow S,
\]
whose fibres are supermanifolds.

It is also useful to describe a supermanifold through its functor of points.  Define
\[
\Lambda_L(\br)=\Lambda^*(\br^L)
\]
and let
\[
\Lambda(\br)=\lim_{L\to\infty}\Lambda_L(\br).
\]
The Grassmann algebra decomposes into its even and odd parts,
\[
\Lambda(\br)
=
\Lambda^0(\br)\oplus\Lambda^1(\br),
\]
also called its \emph{bosonic} and \emph{fermionic} parts.  We denote the purely odd superpoint by
\[
\mathbb{A}^{0|\bullet}_{\br}
=
(\{\mathrm{pt}\},\Lambda(\br)).
\]
The topological manifold $M$ given by the reduced space of $\widehat M$ sits naturally inside the supermanifold $\widehat M$.
\[
\begin{tikzcd}
M \arrow[r] \arrow[d]
    & \widehat M \arrow[d] \\
\{\mathrm{pt}\} \arrow[r]
    & \mathbb{A}^{0|\bullet}_{\br}.
\end{tikzcd}
\]
The right-hand vertical map represents a family of points of $\widehat M$ parametrised by the odd superpoint, while the lower left corner recovers the ordinary points of its reduced space.

For example, the $\mathbb{A}^{0|\bullet}_{\br}$-points of super Euclidean space are
\[
\br^{m|n}(\mathbb{A}^{0|\bullet}_{\br})
=
\left\{
(z_1,\ldots,z_m\mid\theta_1,\ldots,\theta_n)
\;\middle|\;
z_i\in\Lambda^0(\br),\
\theta_j\in\Lambda^1(\br)
\right\}.
\]
Thus the functor-of-points description replaces the formal odd coordinates of the locally ringed-space description by odd elements of a Grassmann algebra.

\subsubsection{Super hyperbolic surfaces}

The supergroup $\operatorname{OSp}(1|2)$ is a supermanifold with reduced space $SL(2,\br)$.  
\[
\begin{tikzcd}
SL(2,\br) \arrow[r] \arrow[d]
    & \operatorname{OSp}(1|2) \arrow[d] \\
\{\mathrm{pt}\} \arrow[r]
    & \mathbb{A}^{0|\bullet}_{\br}.
\end{tikzcd}
\]
Its $\mathbb{A}^{0|\bullet}_{\br}$-points may be written
\[ \text{OSp}(1|2)=\left\{ \left(\left.\begin{array}{c|c}\begin{array}{cc}a\quad&\quad b\\c\quad&\quad d\end{array}&\begin{array}{c}\alpha\\\beta\end{array}\\ \hline a\beta-c\alpha\quad b\beta-d\alpha&1-\alpha\beta\end{array}\right)\right\vert
\begin{array}{cc}a,b,c,d\in\Lambda_0,&  \alpha,\beta\in \Lambda_1\\
ad-bc=1+\alpha\beta&\end{array}\hspace{-1mm}\right\}.\]

Define the super hyperbolic plane by
\[
\widehat{\bh}
=
\left\{
(z\mid\theta)\in\bc^{1|1}(\mathbb{A}^{0|\bullet}_{\br})
\;\middle|\;
\operatorname{Im}z^\#>0
\right\}
\]
where $z=z^\#+z_{\mathrm{nil}}$ for $z_{\mathrm{nil}}$ nilpotent defines the {\em body} $z^\#$ of $z$. 

There is an action of $\operatorname{OSp}(1|2)$ on $\widehat{\bh}$ extending the usual action of $PSL(2,\br)$ on $\bh$ by M\"obius transformations.  Explicitly,
\[
(z\mid\theta)
\longmapsto
\left(
\frac{az+b}{cz+d}
+
\theta\frac{\gamma z+\delta}{(cz+d)^2}
\;\middle|\;
\frac{\gamma z+\delta}{cz+d}
+
\frac{\theta}{cz+d}
\right),
\]
where
\[
\gamma=a\beta-c\alpha,
\qquad
\delta=b\beta-d\alpha.
\]
Taking the body of the even coordinates and setting the odd coordinates to zero defines the reduced $SL(2,\br)$-valued representation associated to an $\operatorname{OSp}(1|2)$-valued representation.  
A discrete subgroup of $\operatorname{OSp}(1|2)$ is called \emph{Fuchsian} if its image in $SL(2,\br)$ is Fuchsian.  A Fuchsian representation
\[
\pi_1(\Sigma)
\longrightarrow
\operatorname{OSp}(1|2)
\]
defines a \emph{super hyperbolic surface}
\[
\begin{tikzcd}
\Sigma \arrow[r] \arrow[d]
    & \widehat\Sigma \arrow[d]
      \makebox[0pt][l]{\,$=\widehat{\bh}/\pi_1(\Sigma)$} \\
\{\mathrm{pt}\} \arrow[r]
    & \mathbb{A}^{0|\bullet}_{\br}
\end{tikzcd}
\]
with reduced space given by the ordinary hyperbolic surface $\Sigma=\bh/\pi_1(\Sigma).$
\subsubsection{Super volumes}

A supermanifold possesses both even and odd directions, and integration over the odd directions is defined by Berezin integration.  In local coordinates
\[
(x_1,\ldots,x_m\mid\theta_1,\ldots,\theta_n),
\]
a super measure is locally of the form
\[
\mu
=
f(x,\theta)\,
dx_1\cdots dx_m\,
d\theta_1\cdots d\theta_n,
\]
where Berezin integration in the odd variables extracts the coefficient of
$\theta_1\cdots\theta_n$.  The super volume is obtained by integrating such a super measure over the supermanifold.  In particular, the super Weil--Petersson volumes arise by integrating the natural super analogue of the Weil--Petersson measure over moduli spaces of super hyperbolic surfaces.

As in ordinary hyperbolic geometry, super hyperbolic surfaces can be decomposed into pairs of pants.  A super hyperbolic pair of pants is determined by three even length coordinates together with two odd coordinates
\[
(x,y,z\mid\alpha,\beta).
\]
Stanford and Witten \cite[D.44]{SWiJTG} derived the function $D(x,y,z)$ defined in \eqref{kerD} via analysing  super hyperbolic pairs of pants.  They produced a super McShane identity which they used in an analogous way to Mirzakhani's methods in \cite{MirSim} to derive the recursion in Theorem~\ref{NSrec}. 
In \cite{HPZSup}, Huang, Penner and Zeitlin prove a super McShane identity for $(g,n)=(1,1)$ by a different method, via a generalisation of Penner coordinates.  The odd parameters enter both the super McShane identity and, through Berezin integration, the corresponding super volume recursion.   The same heuristic argument should extend to general $s$.  To make this argument rigorous, one must in particular justify the vanishing of the volumes of moduli spaces of hyperbolic surfaces having positive-length geodesic boundary components of Ramond type, since such boundary components arise when decomposing into super hyperbolic pairs of pants.

\subsubsection{Volumes of Neveu--Schwarz components}
As described in \ref{spincomp}, 
${\cal M}_{g,\sigma}^{\rm spin}\subset{\cal M}_{g,n}^{\rm spin}$
is connected except in the Neveu--Schwarz case $\sigma=\{1^n\}$ when it consists of two connected components denoted even and odd.  Each component of ${\cal M}_{g,n}^{\rm spin}$ defines a measure on ${\cal M}_{g,n}$ via pushforward of a natural spin measure, and we choose to add the measures obtained from the two Neveu--Schwarz components.  Stanford and Witten \cite{SWiJTG} describe give a matrix model argument to show that these two components contribute equally to the pushforward measure, i.e. their volumes are equal.  This is unexpected since for any given genus $g$ hyperbolic surface, there are $2^g(2^{g-1}+1)$ even spin structures, and $2^g(2^{g-1}-1)$ odd spin structures, of Neveu--Schwarz type.  This defines a degree $2^g(2^{g-1}+1)$, respectively $2^g(2^{g-1}-1)$, cover of ${\cal M}_{g,n}$ leading one to expect the even component to have a greater volume than the odd component.   Moreover, in general different components have different volumes which can be seen via calculation.    
\begin{exercise} {\bf Open problem.}
Prove that the odd and even Neveu--Schwarz components have the same volume.
\end{exercise}
A calculation of the volumes of the odd and even Neveu--Schwarz components for genus one and two is given in \ref{evenodd}.

\section{Finiteness of volumes and recursion relations} \label{compact}
The finiteness of the spin measures defined in Section~\ref{spinmeas} is proven by extending the Weil--Petersson form $\omega^{\text{WP}}$ and the Euler form $e(E_{g,n})$ to cohomology classes on a compactification of  ${\cal M}_{g,n}^{\text{spin}}$.  The extension to cohomology classes enables intersection theory methods to produce proofs of the recursion relations.

To describe Wolpert's extension  \cite{WolWei} of $\omega^{\text{WP}}$ to a representative of a cohomology class over $\overline{{\cal M}}_{g,n}$ we briefly introduce natural classes in $H^*(\overline{\mathcal{M}}_{g,n},\mathbb{Q})$.  
There is a natural 
line bundle $L_i=s_i^*\omega_{\overline{\cal U}_{g,n}/
\overline{\mathcal{M}}_{g,n}}$ for each section 
$\overline{\mathcal{M}}_{g,n}\stackrel{s_i}{\longrightarrow}
\overline{\cal U}_{g,n}$ of the universal curve.  Define
$\psi_i:=c_1(L_i)\in H^2(\overline{\mathcal{M}}_{g,n},\mathbb{Q})$, and via the forgetful map $\overline{\modm}_{g,n+1}\stackrel{\pi}{\longrightarrow}\overline{\modm}_{g,n}$, define $\kappa_m:=\pi_*\psi_{n+1}^{m+1}\in H^{2m}(\overline{\mathcal{M}}_{g,n},\mathbb{Q})$.

Wolpert \cite{WolWei} showed that $\omega^{\text{WP}}$ extends as a current to $\overline{{\cal M}}_{g,n}$ which defines the cohomology class 
\[ [\widehat{\omega}^{\text{WP}}]=2\pi^2\kappa_1\in H^2(\overline{{\cal M}}_{g,n},\br).\]   
Mirzakhani \cite{MirSim} used Wolpert's result to produce an extension of $\omega(L_1,...,L_n)$ to $\overline{{\cal M}}_{g,n}$.
\begin{equation}  \label{wpext}
[\widehat{\omega}(L_1,...,L_n)]=2\pi^2\kappa_1+\frac12\sum_{i=1}^n L_i^2\psi_i\in H^2(\overline{{\cal M}}_{g,n},\br).
\end{equation}  
We will now describe the extension of $e(E_{g,n})$ to $\overline{{\cal M}}_{g,n}^{\text{spin}}$ proven in \cite{NorEnu}.

\subsection{Compactification of the moduli space}

In Sections~\ref{spinstr}, \ref{sheafc} and \ref{spinmeas}, we worked over smooth pointed curves $(C,D)$, and defined a vector bundle over their moduli space $E_{g,n}\to{\cal M}_{g,n}^{\text{spin}}$, given in Definition~\ref{obsbun}, with fibres 
\[E_{g,n}|_{[\Sigma]}=H^0(C,\omega_{C}^{3/2}(D)).\]  
A Hermitian metric on $E_{g,n}$, defined in \eqref{hermetric}, together with a holomorphic structure on $E_{g,n}$ produces a Chern connection $A$ with curvature $F_A$, and Euler form $e(E)$ so that $e(E_{g,n})\exp\omega_{WP}$ defines a finite measure on ${\cal M}_{g,n}^{\text{spin}}$.  As usual, $\omega_{WP}$ denotes both the Weil--Petersson form on ${\cal M}_{g,n}$ and its pullback to ${\cal M}_{g,n}^{\text{spin}}$.  The proof of finiteness of the measure uses an extension to the compactification which we describe next.

\subsection{Extension to the compactification}


A \textit{stable twisted curve} $\cc$ is a nodal curve whose irreducible
components are twisted curves and whose nodes have non-trivial isotropy.
It is equipped with a morphism which forgets the orbifold structure
\[
\rho:\cc\longrightarrow C,
\]
to its coarse curve $C$ which is a stable nodal curve.  We say that $ \cc $ is \textit{smooth} if its coarse curve $ C $ is smooth. 

A stable twisted curve with group $G=\bz_2$ has non-trivial isotropy group $\bz_2$ at the labeled points $p_i\in\cc$ and nodes---corresponding to nodes of the coarse curve. All other points have trivial isotropy group.  
The local structure at each nodal point is described by
\[
\{xy = 0\}/\mathbb{Z}_2,
\]
where the $ \mathbb{Z}_2 $-action is given by 
\[
(-1) \cdot (x, y) = (-x, -y).
\]
The moduli space of stable twisted curves with $\mathbb{Z}_{2}$ isotropy, denoted by $\overline{\mathcal{M}}_{g,n}^{(2)}$, parametrises stable curves with a $\mathbb{Z}_{2}$-orbifold structure at their labeled points and nodes.

Spin structures naturally generalise to twisted curves $\cc$ using orbifold line bundles, as described in \ref{orbdle},  now over curves with possible nodes.  An orbifold line bundle $\cf$ is a locally equivariant bundle over the local charts, such that at each nodal point there is an equivariant isomorphism of fibres. 

\begin{definition}  \label{def:modspincomp}
The moduli space of stable spin curves is defined by
\[\overline{{\cal M}}_{g,n}^{\text{spin}}=\{(\cc,\theta,p_1,...,p_n,\phi)\mid \phi:\theta^2\stackrel{\cong}{\longrightarrow}\omega_{\cc}^{\text{log}}\}
\]
where $\theta$ is a  line bundle over a stable, twisted curve $\cc$ with group $\bz_2$, each nodal point and labeled point $p_i$ has isotropy group $\bz_2$, and all other points have trivial isotropy group.
\end{definition}
As in the smooth case, the orbifold line bundle $\theta$ induces a representation $\bz_2\to\bz_2$ at each $p_i$ which produces a natural decomposition into components
\begin{equation}  \label{modspin}
\overline{{\cal M}}_{g,n}^{\text{spin}}=\bigsqcup_{\sigma\in\{0,1\}^n}\overline{{\cal M}}_{g,\sigma}^{\text{spin}}
\end{equation}
where $\sigma_i=0$ if the representation is trivial, or Ramond, and $\sigma_i=1$ if the representation is non-trivial, or Neveu--Schwarz.
 
We again denote by $p$ the natural degree $2^{2g-1}$ map obtained by forgetting the spin structure and the orbifold structure
\[ p:\overline{{\cal M}}_{g,n}^{\text{spin}}\to\overline{{\cal M}}_{g,n}
\]
and denote by $p^\sigma$ its restriction to $\overline{{\cal M}}_{g,\sigma}^{\text{spin}}$.

Denote by $\ce$ the universal spin structure defined over the universal curve  $\overline{\cu}_{g,n}^{\text{spin}}\stackrel{\pi}{\longrightarrow}\overline{{\cal M}}_{g,n}^{\text{spin}}$.  This produces a natural extension of the bundle $E_{g,n}\to{\cal M}_{g,n}^{\text{spin}}$. 
\begin{definition}  
Define the bundle $\widehat{E}_{g,n}:=-R\pi_*\ce^\vee\hspace{-1mm}\to\overline{{\cal M}}_{g,n}^{\text{spin}}$ with fibre $H^1(\cc,\theta^{\vee})$.   
\end{definition} 
The rank is calculated exactly as in the smooth case \eqref{dimh1} using Riemann-Roch to get:
\[\text{rank }\widehat{E}_{g,n}|_{\overline{{\cal M}}_{g,\sigma}^{\text{spin}}}=2g-2+n-\tfrac12R=2g-2+\tfrac12(n+|\sigma|).\]
Put $\widehat{E}_{g,\sigma}=\widehat{E}_{g,n}|_{\overline{{\cal M}}_{g,\sigma}^{\text{spin}}}$ and for $k=2g-2+\tfrac12(n+|\sigma|)$ define
\begin{equation}   \label{Omegadef}
\Omega_{g,\sigma}:=(-1)^n2^{k+1-g}p_*c_k(\widehat{E}_{g,\sigma})\in H^{2k}(\overline{{\cal M}}_{g,n},\bq).
\end{equation}

Finiteness of the measure $\mu^{(m)}(L_1,...,L_n)$ which is a (normalised) pushforward of $e(E_{g,n})\exp\omega_{WP}(L_1,...,L_n)$---see Definition~\ref{defmeas}---is  a consequence of the following expression for its total measure $\widehat{V}_{g,n}^{(m)}(L_1,...,L_n):=\int_{{\cal M}_{g,n+m}} \mu^{(m)}(L_1,...,L_n).$
\begin{thm} \cite{NorEnu}  \label{thvoltheta}
\begin{equation}  \label{voltheta} 
\widehat{V}_{g,n}^{(m)}(L_1,...,L_n)=\int_{\overline{{\cal M}}_{g,n+m}}\hspace{-5mm}\Omega_{g,\{1^n,0^m\}}\exp\left\{2\pi^2\kappa_1+\frac12\sum_{i=1}^n L_i^2\psi_i\right\}
\end{equation}
\end{thm}

\begin{proof}  
We will show that the Hermitian metric \eqref{hermetric} on $E_{g,n}$ extends smoothly across the nodes to define a Hermitian metric on $\widehat{E}_{g,n}$.  The Chern form of the Hermitian metric on $\widehat{E}_{g,n}$ is used to define the Euler class of $\widehat{E}_{g,n}$ which represents the extension of the Euler form $e(E_{g,n})$.  Over the compact space $\overline{{\cal M}}_{g,n}^{\text{spin}}$, the Euler class is independent of the choice of connection on $\widehat{E}_{g,n}$, hence we can use $c_{\text{top}}(\widehat{E}_{g,\sigma})$.  Integrating its pushforward $\Omega_{g,\{1^n,0^m\}}$, defined in \eqref{Omegadef}, against the extension \eqref{wpext} of
$\omega(L_1,\ldots,L_n)$ gives \eqref{voltheta}.  The remainder of the proof gives the extension of the Hermitian metric which was proven in \cite{NorEnu} for Neveu--Schwarz nodes, but easily extends to allow Ramond nodes.

The proof of the extension of the Hermitian metric at Neveu--Schwarz nodes uses the same estimate \eqref{locest} used to prove existence of the integral \eqref{hermetric}.  The point is that $3/2$-differentials have the same pole behaviour at nodes and at labeled points. This is in contrast to quadratic differentials which have higher order poles at nodes than at labeled points, leading to divergence of the Weil--Petersson metric on $\overline{{\cal M}}_{g,n}$.

 The points in the fibre of $\widehat{E}_{g,n}^\vee$ given by elements of $H^0(C,\omega_C^{3/2}(D))$ have the same simple pole behaviour at nodes and at labeled points.  The pole at a node is present if the behaviour at the node is Neveu--Schwarz and removable if the behaviour at the node is Ramond.  Thus the estimate \eqref{locest} applies also at nodes to prove that the Hermitian metric on $H^0(C,\omega_C^{3/2}(D))$ is well-defined when $C$ is nodal.  
The conclusion is that the Hermitian metric on $E_{g,n}$ extends to a Hermitian metric on $\widehat{E}_{g,n}^\vee$.  Furthermore, it extends to a smooth Hermitian metric on $\widehat{E}_{g,n}^\vee$ because the hyperbolic metric $h$ varies smoothly outside of nodes and has a canonical form around nodes, and the Hermitian metric is defined via an integral over $1/\sqrt{h}$ times smooth sections. 

At Ramond nodes, the corresponding section is holomorphic, rather than having a simple pole, at the two points lying over the node, so the convergence estimate is immediate.  
\end{proof}

\subsection{Boundary strata}  \label{bdystr}
The pushforward of the measure from the components of ${\cal M}_{g,n}^{\rm spin}$ to ${\cal M}_{g,n}$ is important in the proofs of Theorems~\ref{NSrec} and \ref{thmsup}.  Although it may seem more natural to work with the measure over ${\cal M}_{g,n}^{\rm spin}$ rather than its pushforward, the structure of the cohomology classes defined by the extension of the measure is better understood on $\overline{{\cal M}}_{g,n}$, in particular, via a description of its boundary strata via stable graphs. 

\begin{definition} \label{stabgr}
A \emph{stable graph} $\Gamma$ of type $(g,n)$ consists of a connected graph
with $n$ labeled legs, together with a genus $g(v)\geq 0$ assigned to each
vertex $v\in V(\Gamma)$, satisfying
\[
g=h^1(\Gamma)+\sum_{v\in V(\Gamma)}g(v)
\]
and the stability condition $2g(v)-2+n(v)>0$ at every vertex, where $n(v)$ is the valence of $v$, including legs.
The set of half-edges is denoted by $H(\Gamma)$, with
\[
v:H(\Gamma)\longrightarrow V(\Gamma)
\]
sending each half-edge to its incident vertex.  The $n$ labeled legs are
half-edges, and each edge $e\in E(\Gamma)$ consists of two half-edges,
denoted $e_+$ and $e_-$.
\end{definition}

\begin{exercise}
List the stable graphs of type $(g,n)=(1,2)$.  Equip the stable graphs with spin structures in the two cases when the labeled points are Neveu--Schwarz, respectively Ramond.
\end{exercise}

A stable graph $\Gamma$ determines a closed stratum which is the
closure of the locus of stable curves whose dual graph is $\Gamma$
\[\overline{{\cal M}}_{\Gamma}\subset\overline{{\cal M}}_{g,n}
\]
with codimension equal to the number of edges of $\Gamma$.  To each vertex $v$ of $\Gamma$ associate a smooth curve of genus $g(v)$ with $n(v)$ marked points.  The marked points corresponding to the two half-edges of each edge of $\Gamma$ are identified to form a node.  This gives a morphism
\[
\phi_\Gamma:
\prod_{v\in V(\Gamma)}
\overline{\cal M}_{g(v),n(v)}
\longrightarrow
\overline{\cal M}_{g,n},
\]
with image a finite $\operatorname{Aut}(\Gamma)$ cover of $\overline{\cal M}_\Gamma$.


The forgetful map which forgets the last point
\[
\overline{{\cal M}}_{g,n+1}\stackrel{\pi}{\longrightarrow}\overline{{\cal M}}_{g,n}.
\]
acts on stable graphs by removing a labeled edge and contracting any unstable vertices.  

\subsubsection{Chiodo classes}
The pushforward of the {\em total} Chern class of $E_{g,n}$
defines classes for $2g-2+n>0$ known as Chiodo classes which define a cohomological field theory on the vector space $\bc^2$. 
Using Chiodo's calculation \cite{ChiTow} of the Chern character of $E_{g,n}$ and an expression for the total Chern class in terms of the Chern character, Janda,  Pandharipande, Pixton and Zvonkine \cite[Corollary 4]{JPPZDou} gave the following formula for the pushforward of the total Chern class as a weighted sum over stable graphs of type $(g,n)$, denoted $G_{g,n}$.  
\begin{align}  \label{JPPZ}
p_*c(E_{g,n})&=\sum_{\Gamma\in G_{g,n}}w(\Gamma)\in H^*(\overline{{\cal M}}_{g,n},\bq)\\
&=\sum_{(\Gamma,\sigma)\in G_{g,n}^{\text{spin}}} \frac{2^{2g-1-h^1(\Gamma)}}{|\text{Aut}(\Gamma)| }\phi_{\Gamma*}\Bigg[
\prod_{v \in V(\Gamma)}\hspace{-2mm}f(v)\prod_{e \in E(\Gamma)}\hspace{-2mm}g(e)\ \ \prod_{i=1}^nh(i)\Bigg]\nonumber
\end{align}
where we sum over all $\sigma:H(\Gamma)\to\{0,1\}$  satisfying $\sigma(e_+)=\sigma(e_-)$ on each edge $e\in E(\Gamma)$, and for any $v\in V(\Gamma)$
\[\sum_{\{\ell\mid v(\ell)=v\}}\hspace{-4mm}\sigma(\ell)\equiv n(v)(\text{mod } 2) .\]
Each choice of $\sigma$ specifies Neveu--Schwarz or Ramond behaviour at each node and labeled point.   The sum over all extensions $\sigma$ of $\sigma(\ell_i)=\sigma_i$ for each labeled leg $\ell_i$, gives $p_*c(E_{g,\sigma})$.

The vertex, edge and leg functions are:
\[f(v)=\exp\sum\limits_{m\geq 1} (-1)^{m}\frac{B_{m+1}(-1/2)}{m(m+1)}\kappa_m(v),
\]
\[ g(e)=\frac{1-\exp\sum\limits_{m \geq 1} (-1)^{m-1} \frac{B_{m+1}(\sigma(e_+)/2)}{m(m+1)} [(\psi_{e_+})^m-(-\psi_{e_-})^m]}{\psi_{e_+} + \psi_{e_-}} 
\]
and
\[h(i)=\exp\sum\limits_{m\geq 1}(-1)^{m-1} \frac{B_{m+1}(\sigma_i/2)}{m(m+1)} \psi^m_{i}
\]
where $B_m$ is a Bernoulli polynomial.

The top degree terms of \eqref{JPPZ} give an expression for $\Omega_{g,\sigma}$, defined in \eqref{Omegadef}, as a sum over weighted graphs.   We  give some explicit calculations of this weighted sum in \ref{evenodd}.

\subsubsection{Integrability}
The behaviour of the pushforward cohomology classes, corresponding to the pushforward measure is demonstrated elegantly in the Neveu--Schwarz case
\[\Theta_{g,n}:=\Omega_{g,\{1^n\}}\in H^{2g-2+n}(\overline{{\cal M}}_{g,n},\bq).\] 
Write $\Theta_\Gamma:=\otimes_{v\in V(\Gamma)}\Theta_{g(v),n(v)}\in\prod_{v\in V(\Gamma)}\overline{\cal M}_{g(v),n(v)}$.  The
intersection numbers of $\Theta_{g,n}$, which produce the volumes $\widehat{V}_{g,n}^{(0)}(L_1,...,L_n)$ via Theorem~\ref{thvoltheta}, are characterised by  four properties \cite{NorNew}:
\begin{enumerate}[(i)]
\setlength{\itemindent}{20pt}
\item $\Theta_{g,n}\in H^*(\overline{{\cal M}}_{g,n},\bq)$ is of pure degree, \label{pure}
\item $\phi_\Gamma^*\Theta_{g,n}=\Theta_\Gamma$,  \label{glue}
\item $\Theta_{g,n+1}=\psi_{n+1}\cdot\pi^*\Theta_{g,n}$,  \label{forget}
\item  $\Theta_{1,1}\neq 0$.  \label{base}
\end{enumerate}
From these properties, the intersection numbers are uniquely determined. This characterisation does not use spin structures, suggesting possible alternative descriptions of the classes.  One such description is given in \cite{CGGRel,KNoPol}.  

The intersection numbers, uniquely determined by (i)-(iv) above, assemble to produce a tau function of the KdV hierarchy.
\begin{thm}[Chidambaram, N.; Garcia-Failde, E. and Giacchetto, A, \cite{CGGRel}]  \label{CGG}
The  function
\begin{equation} \label{BGWTheta}
Z^{BGW}(\hbar,t_0,t_1,...)=\exp\sum_{g,n,\vec{k}}\frac{\hbar^{g-1}}{n!}\int_{\overline{{\cal M}}_{g,n}}\Theta_{g,n}\cdot\prod_{j=1}^n\psi_j^{k_j}\prod t_{k_j}\end{equation} 
is the Br\'ezin-Gross-Witten tau function of the KdV hierarchy. \label{thetatau}
\end{thm}
Theorem~\ref{CGG} was proven in \cite{CGGRel} by showing that the right hand side of \eqref{BGWTheta} satisfies Virasoro constraints.  This was used in \cite{NorEnu} to prove Theorem~\ref{NSrec} which encodes a Virasoro-like structure on the collection of volumes.  The left and right hand sides of \eqref{BGWTheta} can both be encoded into topological recursion spectral curves.  Via generalised topological recursion \cite{ABDKSDeg,BCGSThe} good behaviour of the limit of spectral curves then produces a second proof of Theorem~\ref{CGG}.

Theorem~\ref{thmsup} was proven via the following generalisation of Theorem~\ref{CGG} to allow Ramond points, which uses the generalised Br\'ezin-Gross-Witten tau function depending on a parameter $s$.  
\begin{thm}[Alexandrov, N., \cite{ANoSup}] \label{BGW=spin} 
\[Z^{BGW}(\hbar,s,t_0,t_1,...)=\exp\sum_{g,n}\frac{\hbar^{g-1}}{n!}\sum_{\vec{k}\in\bn^n}\sum_{m=0}^\infty\frac{s^m}{m!}\int_{\overline{\modm}_{g,n+m}}\hspace{-5mm}\Omega_{g,n+m}^{(1^n,0^m)}.\prod_{i=1}^n\psi_i^{k_i}t_{k_i}.\]
\end{thm}
The intersection numbers of $\Omega_{g,n+m}^{(1^n,0^m)}$, defined in \eqref{Omegadef}, are related in \eqref{voltheta} to the volume polynomial in Theorem~\ref{thmsup}.  Then Virasoro structure of $Z^{BGW}(\hbar,s,t_0,t_1,...)$ produces the recursion \eqref{recsup}.

The methods of proof of Theorems~\ref{NSrec} and \ref{thmsup} both rely on intersection calculations over $\overline{{\cal M}}_{g,n}$. We leave further details to the references above.  Essentially the algebro-geometric methods verify the truth of the statements, but somehow miss the reason.  An explanation for the structure of Theorems~\ref{NSrec} and \ref{thmsup} is supplied by the supergeometric methods, though lacking complete rigour.  The next calculations emphasise this further, where  algebro-geometric methods verify symmetry without giving an underlying reason.

\subsubsection{Even/odd symmetry in volume calculations}  \label{evenodd}

We will apply the weighted sum \eqref{JPPZ} to calculate volumes of the Neveu--Schwarz components in genus one and genus two, and in the process verify that the odd and even volumes are equal in these cases.  It seems that there is a non-obvious symmetry that is hidden from the calculations since contributions from various stable graphs lack symmetry.\\

\noindent $\boxed{(g,n)=(1,1)}$\quad 
There are two stable graphs:
 \begin{center}
\begin{tikzpicture}
\draw (0,-1) node        {$\Gamma_1$};
\draw (0,0) node [shape=circle,draw]        {1};
\draw (.3,0)--(.7,0);
\draw (4,-1) node        {$\Gamma_2$};
\draw (4,0) node [shape=circle,draw]        {0};
\draw (4.3,0)--(4.7,0);
\draw (3.9,.3) arc (20:336:.8);
\end{tikzpicture}
\end{center}
Their contributions to $p_*c_1(E_{1,1})$ using \eqref{JPPZ} are
$$w(\Gamma_1)=2(-\frac{11}{24}\kappa_1-\frac{1}{24}\psi_1)=-\frac{1}{24},\quad w(\Gamma_2)=\frac{1}{2}\left(-\frac{1}{12}+\frac{1}{24}\right)
$$
Hence 
$$p_*c_1(E_{1,1})=w(\Gamma_1)+w(\Gamma_2)=-\frac{1}{24}-\frac{1}{48}=-\frac{1}{16}=-2^{-1}\widehat{V}_{1,1}(L)$$
which uses the factor $(-1)^n2^{1-g-n}=-2^{-1}$ in the definition of the measure.

The odd contributions to $p_*c_1(E_{1,1})$ have weights
$$w_{\text{odd}}(\Gamma_1)=\frac{1}{4}\left(-\frac{1}{24}\right),\quad w_{\text{odd}}(\Gamma_2)=\frac{1}{2}\left(\frac{1}{2}\left(-\frac{1}{12}\right)+0\cdot\frac{1}{24}\right).
$$
This uses the dependence of the weights on the Neveu--Schwarz or Ramond behaviour at nodes via $\sigma:H(\Gamma)\to\{0,1\}$.  For the calculation above, of the four spin structures on the family of elliptic curves, two degenerate to a Neveu--Schwarz node and two degenerate to a Ramond node.  In particular, the odd spin structure degenerates to a Ramond node which does not contribute to the $\frac12(\frac{1}{24})$ term in $w(\Gamma_2)$. 
Hence
$$p_*c_1(E_{1,1}|_{\text{odd}})=w_{\text{odd}}(\Gamma_1)+w_{\text{odd}}(\Gamma_2)=\frac{1}{4}\left(-\frac{1}{24}\right)-\frac{1}{2}\frac{1}{2}\frac{1}{12}=-\frac{1}{32}
$$
so we see that
\[\widehat{V}_{1,1}^{\text{odd}}(L)=\widehat{V}_{1,1}^{\text{even}}(L).\]  
A relationship between volumes, known as the dilaton equation \cite{NorEnu}, allows us to deduce more generally
$$\widehat{V}_{1,n}^{\text{odd}}(L_1,...,L_n)=\widehat{V}_{1,n}^{\text{even}}(L_1,...,L_n).$$

\noindent $\boxed{(g,n)=(2,0)}$\quad There are five stable graphs:
 \begin{center}
\begin{tikzpicture}[scale=0.8]
\draw (0,-1) node        {$\Gamma_1$};
\draw (2,-1) node        {$\Gamma_2$};
\draw (5,-1) node        {$\Gamma_3$};
\draw (8.5,-1) node        {$\Gamma_4$};
\draw (11,-1) node        {$\Gamma_5$};
\draw (0,0) node [shape=circle,draw]        {2};
\draw (1.5,0) node [shape=circle,draw]        {1};
\draw (2.5,0) node [shape=circle,draw]        {1};
\draw (1.8,0)--(2.2,0);
\draw (5,0) node [shape=circle,draw]        {0};
\draw (5.3,0)--(5.7,0);
\draw (6,0) node [shape=circle,draw]        {1};
\draw (4.9,.3) arc (20:336:.8);
\draw (5,0) node [shape=circle,draw]        {0};
\draw (5.3,0)--(5.7,0);
\draw (6,0) node [shape=circle,draw]        {1};
\draw (8.5,0) node [shape=circle,draw]        {1};
\draw (8.4,.3) arc (20:336:.8);
\draw (11,0) node [shape=circle,draw]        {0};
\draw (10.9,.3) arc (20:336:.8);
\draw (11.1,.3) arc (160:-157:.8);
\end{tikzpicture}
\end{center}
Their contributions to $p_*c_2(E_{2})$ using \eqref{JPPZ} are
$$w(\Gamma_1)=-\kappa_2+\frac{121}{144}\kappa_1^2,\quad w(\Gamma_2)=-\frac{23}{36}\kappa_2+\frac{23}{144}\kappa_1^2,\quad w(\Gamma_3)=-\frac{179}{72}\kappa_2+\frac{11}{12}\kappa_1^2
$$
$$w(\Gamma_4)=-\frac13\kappa_2+\frac{1}{12}\kappa_1^2,\quad w(\Gamma_5)=-\frac{19}{24}\kappa_2+\frac{1}{4}\kappa_1^2
$$
Hence 
$$p_*c_2(E_{2})=w(\Gamma_1)+w(\Gamma_2)+w(\Gamma_3)+w(\Gamma_4)+w(\Gamma_5)=-\frac{21}{4}\kappa_2+\frac{9}{4}\kappa_1^2
$$
and 
$$\widehat{V}_2=2\int_{\overline{{\cal M}}_2}p_*c_2(E_{2})\exp(\omega^{\text{WP}})=4\pi^2\int_{\overline{{\cal M}}_2}p_*c_2(E_{2})\kappa_1=\frac{3\pi^2}{64}
$$
where the coefficient $2=(-1)^n2^{g-1+n}$ comes from the factor in the definition of the measure.

The odd contributions to $p_*c_1(E_2)$ have weights
$$w_{\text{odd}}(\Gamma_1)=\frac38w(\Gamma_1),\quad w_{\text{odd}}(\Gamma_2)=\frac38w(\Gamma_2),\quad w_{\text{odd}}(\Gamma_3)=-\frac{169}{96}\kappa_2+\frac{11}{16}\kappa_1^2
$$
$$w_{\text{odd}}(\Gamma_4)=-\frac14\kappa_2+\frac{1}{16}\kappa_1^2,\quad w_{\text{odd}}(\Gamma_5)=0
$$
so
\begin{align*}
p_*c_1(E_{2}|_{\text{odd}})&=w_{\text{odd}}(\Gamma_1)+w_{\text{odd}}(\Gamma_2)+w_{\text{odd}}(\Gamma_3)+w_{\text{odd}}(\Gamma_4)+w_{\text{odd}}(\Gamma_5)\\
&=-\frac{21}{8}\kappa_2+\frac{9}{8}\kappa_1^2
\end{align*}
and again we see that $\widehat{V}_2^{\text{odd}}=\widehat{V}_2^{\text{even}}$.  By the dilaton equation and intersection with a $\psi$ class,
$$\widehat{V}_{2,n}^{\text{odd}}(L_1,...,L_n)=\widehat{V}_{2,n}^{\text{even}}(L_1,...,L_n).$$

Hence the odd and even contributions to the volume are equal in genus one and genus two which is unexpected from the detailed calculations.  In the $(1,1)$ case, one might expect the even and odd components to contribute $3/4$, respectively $1/4$, to the volume, and indeed the one vertex stable graph contribution is split $3/4$ and $1/4$.  Similarly, in the $(2,0)$ case, one might expect the even and odd components to contribute $10/16$, respectively $6/16$, to the volume, and indeed the first two stable graph contributions are split $10/16$ and $6/16$.  (If the graph is a tree, then the Neveu--Schwarz condition is forced on internal edges, and even and odd components contribute $2^{g-1}(2^g+1), 2^{g-1}(2^g-1)$ as expected.)  The contributions from other stable graphs are a sum of contributions over different choices of Neveu--Schwarz/Ramond behaviour on  edges, and it is entirely non-trivial that the final even and odd component contributions turn out to be equal.

\end{document}